\documentclass{amsart}

\usepackage{stmaryrd}
\usepackage{amssymb}

\newtheorem{theorem}{Theorem}
\newtheorem*{maintheorem*}{Main Theorem}
\newtheorem{lemma}[theorem]{Lemma}
\newtheorem{proposition}[theorem]{Proposition}
\newtheorem{corollary}[theorem]{Corollary}
\newtheorem{claim}[theorem]{Claim}

\theoremstyle{definition}
\newtheorem{definition}[theorem]{Definition}
\newtheorem{example}[theorem]{Example}

\theoremstyle{remark}
\newtheorem{remark}[theorem]{Remark}

\newcommand{\var}{\mathop{\mathrm{var}}}
\newcommand{\essinf}{\mathop{\mathrm{ess\,inf}}}
\newcommand{\ran}{\mathop{\mathrm{ran}}}

\newcommand{\Proj}[2]{\mathrm{Pr}_{#1\sslash #2}}

\title[A semi-invertible Oseledets Theorem]{A semi-invertible Oseledets Theorem with applications to transfer operator
cocycles}

\author{Gary Froyland}
\email{g.froyland@unsw.edu.au}
\author{Simon Lloyd}
\email{s.lloyd@unsw.edu.au}
\author{Anthony Quas}
\email{aquas@uvic.ca}
\address[Froyland and Lloyd]{School of Mathematics and Statistics,
University of New South Wales, Sydney NSW 2052, Australia}
\address[Quas]{Department of Mathematics and Statistics, University of
Victoria, Victoria BC, Canada V8W 3R4}

\begin{document}

\begin{abstract}
Oseledets' celebrated Multiplicative Ergodic Theorem (MET) [V.I.
Oseledec, \emph{A multiplicative ergodic theorem. {C}haracteristic
{L}japunov, exponents of dynamical systems}, Trudy Moskov. Mat. Ob\v s\v
c. \textbf{19} (1968), 179--210.] is concerned with the exponential
growth rates of vectors under the action of a linear cocycle on
$\mathbb{R}^d$. When the linear actions are invertible, the MET
guarantees an almost-everywhere pointwise \emph{splitting} of
$\mathbb{R}^d$ into subspaces of distinct exponential growth rates
(called Lyapunov exponents).  When the linear actions are non-invertible,
Oseledets' MET only yields the existence of a \emph{filtration} of
subspaces, the elements of which contain all vectors that grow no faster
than exponential rates given by the Lyapunov exponents.  The authors
recently demonstrated [G. Froyland, S. Lloyd, and A. Quas, \emph{Coherent
structures and exceptional spectrum for {P}erron--{F}robenius cocycles},
Ergodic Theory and Dynam. Systems (to appear).] that a splitting over
$\mathbb{R}^d$ is guaranteed \emph{without} the invertibility assumption
on the linear actions. Motivated by applications of the MET to cocycles
of (non-invertible) transfer operators arising from random dynamical
systems, we demonstrate the existence of an Oseledets splitting for
cocycles of quasi-compact non-invertible linear operators on Banach
spaces.
\end{abstract}

\maketitle

\section{Introduction}

Oseledets-type ergodic theorems deal with dynamical systems
$\sigma\colon \Omega\to\Omega$ where for each $\omega\in\Omega$ there is
an operator (or in the original Oseledets case
a matrix) $\mathcal L_\omega$ acting on a linear space $X$. One then
studies the properties of the operator $\mathcal L^{(n)}_\omega=\mathcal
L_{\sigma^{n-1}\omega}\circ\cdots \circ \mathcal L_\omega$, giving an
$\omega$-dependent decomposition of $X$ into subspaces with a hierarchy
of expansion properties.

Prior to the previous work of the current authors, \cite{FLQ}, all of the
Oseledets-type theorems in the literature split into two cases according
to the hypotheses: the invertible and non-invertible cases.

\begin{description}
\item[Invertible case]
In this case the base dynamical system $\sigma$ is assumed to be
invertible and the operators $\mathcal L_\omega$ are assumed to be invertible
(or in some cases just injective). Integrability conditions may be
imposed on $\|\mathcal L^{-1}_\omega\|$.

The conclusion here is that the space $X$ admits an invariant
\emph{splitting} $E_1(\omega)\oplus E_2(\omega)\oplus \cdots$, finite
or countable, possibly with a `remainder' in the
infinite-dimensional case. Non-zero vectors in $E_i(\omega)$ expand
exactly at rate $\lambda_i$.

\item[Non-invertible case] In the non-invertible case no assumptions
  are made about invertibility of the base nor about injectivity of
  the operators. The weaker conclusion here is that $X$ admits an
  invariant \emph{filtration} $V_1(\omega)\supset
  V_2(\omega)\supset\cdots$ such that vectors in $V_i(\omega)\setminus
  V_{i+1}(\omega)$ expand at rate $\lambda_i$.

\end{description}

The conclusion in the invertible case may be seen to be much
stronger as one obtains an invariant family of complements to
$V_{i+1}(\omega)$ in $V_i(\omega)$. These are in general
finite-dimensional so that one `sees the vectors responsible for
$\lambda_i$ expansion'. This is of considerable importance in
applications.

Our principal contribution here is to focus on the
\emph{semi-invertible} case. Here assumptions are made about the
invertibility of the base transformation, but there are no assumptions
about invertibility or injectivity of the operators $\mathcal
L_\omega$. In spite of this we are able to show that one can obtain an
invariant splitting rather than the weaker invariant filtration, for
the setting investigated by Thieullen \cite{Thieullen} where the
random compositions have some quasi-compactness properties. In
\cite{FLQ} we obtained an analagous result for the original Oseledets
setting of matrices acting on $\mathbb R^d$.

\subsection{Set-up}

Let $(\Omega,\mathcal{F},\mathbb{P})$ be a probability space and
$(X,\|\cdot\|)$ a Banach space. A \emph{random dynamical system} is a
tuple $\mathcal R=(\Omega,\mathcal{F},\mathbb{P},\sigma,X,\mathcal L)$,
where $\sigma$ is an invertible measure-preserving transformation of
$(\Omega,\mathcal F,\mathbb{P})$, called the \emph{base} transformation,
and $\mathcal L:\Omega\to L(X,X)$ is a family of bounded linear maps of
$X$, called the \emph{generator}. We will later impose suitable
measurability conditions on $\mathcal L$.

For notational convenience, we write $\mathcal L(\omega)$ as $\mathcal
L_\omega$. A random dynamical system defines a \emph{cocycle}
$\mathbb{N}\times\Omega\to L(X,X)$:
\begin{equation}\label{eqn:cocycle}
(n,\omega)\mapsto \mathcal L^{(n)}_\omega:= \mathcal
L_{\sigma^{n-1}\omega}\circ\cdots\circ \mathcal
L_{\sigma\omega}\circ\mathcal L_{\omega}.
\end{equation}
We define the \emph{Lyapunov exponent in direction $v$},
$\lambda(\omega,v)$, by
\begin{equation}\label{eqn:exponent}
\lambda(\omega,v) = \limsup_{n\to\infty} \frac{1}{n}\log \|\mathcal
L^{(n)}_\omega v\|,\quad \omega\in\Omega,\ \ v\in X.
\end{equation}
Lyapunov exponents have the following well-known properties. For all
$\omega\in\Omega$, $u,v \in X$ and $\alpha\neq 0$:
\begin{enumerate}
\item[(i)] $\lambda(\omega,0)=-\infty$;
\item[(ii)] $\lambda(\omega,\alpha v) = \lambda(\omega,v)$;
\item[(iii)] $\lambda(\omega,u+v) \leq
\max\{ \lambda(\omega,u),\lambda(\omega,v)\}$
with equality if $\lambda(\omega,u)\neq \lambda(\omega,v)$;
\item[(iv)] $\lambda(\sigma\omega,\mathcal L_\omega v)=\lambda(\omega,v)$.
\end{enumerate}
We call the set $\Lambda(\omega) =
\{\lambda(\omega,v): v\in X\}$ the
\emph{Lyapunov spectrum}.
For $\alpha\in\mathbb{R}$, the set $\mathcal{V}_\alpha(\omega):=\{v\in X:
\lambda(\omega,v)\leq \alpha\}$ is a linear subspace of $X$, $\mathcal
L_\omega\mathcal{V}_\alpha(\omega)\subset
\mathcal{V}_\alpha(\sigma\omega)$ and if $\alpha'<\alpha$, then
$\mathcal{V}_{\alpha'}(\omega)\subset \mathcal{V}_\alpha(\omega)$.
For each $\omega\in\Omega$, the quantity $\lambda(\omega)$ is defined by
\begin{equation}\label{eqn:lambda}
\lambda(\omega) = \limsup_{n\to\infty} \frac{1}{n}\log \|\mathcal
L^{(n)}_\omega\|.
\end{equation}

\begin{definition}
Let $\mathcal R=(\Omega,\mathcal{F},\mathbb{P},\sigma,X,\mathcal L)$ be a
random dynamical system.
\begin{itemize}
\item We say that $\mathcal{R}$ is
\emph{quasi-compact} if for almost every $\omega$ there is an
$\alpha<\lambda(\omega)$ such that $\mathcal V_{\alpha}(\omega)$ is
finite co-dimensional. Of particular interest is the infimal $\alpha$
with this property. We call this quantity $\alpha(\omega)$.
\item For each isolated Lyapunov exponent $r\in\Lambda(\omega)$, let
$\epsilon_r>0$ be small enough that
$\Lambda(\omega)\cap(r-\epsilon_r,r)=\emptyset$. If the codimension $d$
of $\mathcal{V}_{r-\epsilon_r}(\omega)$ in $\mathcal{V}_r(\omega)$ is
finite, then we say $r$ is a Lyapunov exponent of \emph{multiplicity}
$d$.
\item The Lyapunov exponents greater than $\alpha(\omega)$ are said to be
\emph{exceptional}. As they are isolated, the exceptional Lyapunov
exponents $\{\lambda_i(\omega)\}_{i=1}^{p(\omega)}$ are either finite in
number $(p(\omega)<\infty)$ or else they are countably infinite
$(p(\omega)=\infty)$, accumulating only at $\alpha(\omega)$. We shall
always enumerate the exceptional Lyapunov exponents in decreasing order
$\lambda_1(\omega)>\lambda_2(\omega)>\cdots$. The \emph{exceptional
Lyapunov spectrum},
$\mathrm{EX}(\mathcal{R})(\omega)=
\{(\lambda_i(\omega),d_i(\omega))\}_{i=1}^{p(\omega)}$,
consists of all exceptional Lyapunov exponents paired with their
multiplicities.
\end{itemize}
\end{definition}

In the setting where $\mathbb P$ is ergodic and the generator $\mathcal L$
satisfies suitable measurability conditions,
$\lambda(\omega)=\lambda^*$, $\alpha(\omega)=\alpha^*$ and the
exceptional Lyapunov spectrum will be independent of $\omega$
$\mathbb{P}$-a.e.

If $X$ is a finite dimensional space, then $\alpha(\omega)=-\infty$
for each $\omega$ and so all Lyapunov exponents are exceptional.
Since the sets $\mathcal{V}_\alpha(\omega)$ are subspaces for each
$\alpha\in\mathbb{R}$, the number of Lyapunov exponents counted with
multiplicity is bounded by the dimension of $X$, and so each is
isolated and of finite multiplicity.

We are interested in Banach space analogues of the multiplicative ergodic
theorem. In order to make sense of this it will be necessary to put a
topology on suitable collections of subspaces of Banach spaces. The
Grassmannian $\mathcal G(X)$ of a Banach space $X$ is defined to be the
set of \emph{complemented} closed subspaces $E$ of $X$ (that is, those
for which there is a second closed subspace $F$ with the property that
$X=E\oplus F$). Since every finite dimensional subspace of $X$ is closed
and complemented, the collection of $d$-dimensional subspaces of $X$
forms a subset of $\mathcal{G}(X)$, which we denote by
$\mathcal{G}_d(X)$. Also, finite codimensional subspaces are necessarily
complemented, so the collection of \emph{closed} $c$-codimensional
subspaces of $X$ forms a subset of $\mathcal{G}(X)$, which we denote by
respectively $\mathcal{G}^c(X)$. More details on the Grassmannian are
given in Section \ref{sec:grass} along with proofs of some basic theorems
concerning Grassmannians that we shall need later.

\begin{definition}
Consider a random dynamical system
$\mathcal{R}=(\Omega,\mathcal{F},\mathbb{P}, \sigma,X,\mathcal L)$ with
ergodic base, and suppose $\mathcal{R}$ is quasi-compact with exceptional
spectrum $\{(\lambda_i,d_i)\}_{i=1}^p$ (where $1\le p\le\infty$). A
\emph{Lyapunov filtration} for $\mathcal{R}$ is a collection of maps
$(V_i:\Omega\to\mathcal{G}^{c_i}(X))_{i=1}^p$,
 such that for all $\omega$ in a full measure $\sigma$-invariant subset
$\Omega'\subset \Omega$ and for each
$i=1,\ldots,p$:
\begin{enumerate}
\item $X=V_1(\omega)\supset \cdots\supset V_i(\omega)\supset V_{i+1}(\omega)$;
\item $\mathcal{V}_{\alpha(\omega)}(\omega)\subseteq \bigcap_{i=1}^p V_i(\omega)$
 with equality if and only if $p=\infty$;
\item $\mathcal L_\omega V_i(\omega) = V_i(\sigma\omega)$;
\item $\lambda(\omega,v) =
\lim_{n\to\infty}(1/n)\log\|\mathcal L^{(n)}_\omega v\| = \lambda_i$
if and only if $v\in V_i(\omega)\backslash V_{i+1}(\omega)$,
\end{enumerate}
where we set $V_{p+1}(\omega):=\mathcal{V}_{\alpha(\omega)}(\omega)$ if
$p<\infty$. An \emph{Oseledets splitting} for $\mathcal{R}$ is a Lyapunov
filtration $(V_i:\Omega\to\mathcal{G}^{c_i}(X))_{i=1}^p$ together with an
additional collection of maps
$(E_i:\Omega\to\mathcal{G}_{d_i}(X))_{i=1}^p$ (with $d_i=c_{i+1}-c_i$),
called \emph{Oseledets subspaces}, such that for all $\omega$ in a full
measure $\sigma$-invariant subset $\Omega'\subset \Omega$ and for each
$i=1,\ldots,p$:
\begin{enumerate}
\setcounter{enumi}{4}
\item $ V_i(\omega)=E_i(\omega)\oplus V_{i+1}(\omega)$;
\item $\mathcal L_\omega E_i(\omega) = E_i(\sigma\omega)$;
\item $\lambda(\omega,v) =
\lim_{n\to\infty}(1/n)\log\|\mathcal L^{(n)}_\omega v\| = \lambda_i$
if $v \in E_i(\omega)\backslash\{0\}$.
\end{enumerate}
\end{definition}

We say a Lyapunov filtration is measurable if the maps
$V_i:\Omega\to\mathcal{G}^{c_i}(X)$ are measurable with respect to the
Borel $\sigma$-algebra on $\mathcal G(X)$ for each $1\leq i\leq p$, where
the topology will be defined in the next section.  We say an Oseledets
splitting is measurable if, in addition, the maps
$E_i:\Omega\to\mathcal{G}_{d_i}(X)$ are measurable.

\subsection{Multiplicative Ergodic Theorems}

The first result on the existence of Lyapunov filtrations and
Oseledets splittings in the finite dimensional setting is the
Multiplicative Ergodic Theorem of Oseledets. Throughout, we define
$\log^+(x)=\max\{0,\log x\}$.

\begin{theorem}[Oseledets \cite{Oseledec}]\label{thm:Oseledec}
Let
$\mathcal{R}=(\Omega,\mathcal{F},\mathbb{P},\sigma,\mathbb{R}^d,\mathcal
L)$ be a random dynamical system with ergodic base, and suppose that the
generator $\mathcal L$ is measurable and $\int \log^+\|\mathcal
L_\omega\|\,\mathrm{d}\mathbb{P} <+\infty$. Then $\mathcal{R}$ admits a
measurable Lyapunov filtration.

Moreover, if the base is invertible, $\mathcal L_\omega$ is invertible
a.\,e. and $\int \log^+\|\mathcal L_\omega^{\pm1}\|\,\mathrm{d}\mathbb{P}
<+\infty$, then $\mathcal{R}$ admits a measurable Oseledets splitting.
\end{theorem}

This situation may be summarized by saying that Oseledets splittings can
be found when the base is invertible and the linear actions in the
cocycle are invertible with bounded inverses, whereas in the
non-invertible linear action cases the theorem only guarantees a Lyapunov
filtration. This situation persisted in all subsequent versions
\cite{Ruelle,Mane,Thieullen} and extensions of the Oseledets theorem, to
our knowledge, until the result stated below by the current authors which
obtained a Oseledets splitting in the \emph{semi-invertible} case where
the base is invertible but the operators themselves are not assumed to be
invertible (or they are invertible but there is no bound on the
logarithmic norms of their inverses).

\begin{theorem}[Froyland, Lloyd and Quas \cite{FLQ}]\label{thm:FLQ}
Let
$\mathcal{R}=(\Omega,\mathcal{F},\mathbb{P},\sigma,\mathbb{R}^d,\mathcal
L)$ be a random dynamical system with an invertible ergodic base, and
suppose $\mathcal L$ is measurable and $\int \log^+\|\mathcal
L_\omega\|\,\mathrm{d}\mathbb{P} <+\infty$. Then $\mathcal{R}$ admits a
measurable Oseledets splitting.
\end{theorem}

\begin{remark}
It is natural to ask whether one can obtain an invariant splitting in the
absence of invertibility of either the base or the operators. In section
\ref{sec:invbase} we show that if the base is non-invertible then even in
the case where the operators are invertible one cannot in general obtain
an invariant splitting.
\end{remark}

The result of Oseledets has been extended by many authors. Of particular
relevance to us are the result of Ruelle \cite{Ruelle} dealing with the
case where $X$ is a Hilbert space and the result of Ma\~{n}\'e
\cite{Mane} on random dynamical systems of compact operators in Banach
spaces. This was subsequently extended to the quasi-compact case by
Thieullen \cite{Thieullen}. Thieullen's result will be stated precisely
in Section \ref{sec:thieullen}. A key requirement for Thieullen's
extension is that the dependence of the operator $\mathcal L_\omega$ on
$\omega$ is required to be $\mathbb{P}$-continuous (the definition
follows in Section \ref{sec:thieullen}) and it is upon this that we
build. It should be pointed out that this is a significant limitation as
many natural random dynamical systems fail to satisfy this condition
(e.g. if $T_\omega$ is a family of Lasota--Yorke maps, it is almost never
the case that their Perron--Frobenius operators depend in a
$\mathbb{P}$-continuous way on $\omega$). A parallel approach was taken
in the recent thesis of Lian \cite{Lian} where the measurability
condition is relaxed to the weaker `strongly measurable' condition
(meaning that for each fixed $x\in X$, the map $\omega\mapsto \mathcal
L_\omega x$ is measurable). The cost (which is again heavy from the point
of view of applications) is that in order to obtain suitable
measurability Lian imposes the condition that the Banach space $X$ be
separable.

Our main theorem is related to Thieullen's theorem in exactly the way
that our theorem from \cite{FLQ} is related to Oseledets' Theorem: it
provides an Oseledets splitting for the category of a quasi-compact
linear action in the semi-invertible case where the base is invertible
without any invertibility assumptions on the operators. We include the
statement here, but defer some relevant definitions to section
\ref{sec:thieullen}.

\begin{maintheorem*}[Theorem \ref{thm:main}]
Let $\Omega$ be a Borel subset of a separable complete metric space,
$\mathcal{F}$ the Borel sigma-algebra and $\mathbb{P}$ a Borel
probabilty. Let $X$ be a Banach space and consider a random dynamical
system $\mathcal{R}=(\Omega,\mathcal{F},\mathbb{P},\sigma,X,\mathcal L)$
with base transformation $\sigma:\Omega\to\Omega$ an ergodic
homeomorphism, and suppose that the generator $\mathcal L:\Omega\to
L(X,X)$ is $\mathbb{P}$-continuous and satisfies
$$
\int \log^+\|\mathcal L_\omega\|\,\mathrm{d}\mathbb{P} <+\infty.
$$
If $\kappa(\omega)<\lambda(\omega)$ (where $\kappa$ is the ``index of
compactness'' of $\mathcal{L}$) for almost every $\omega$, then
$\mathcal{R}$ is quasi-compact and admits a unique
$\mathbb{P}$-continuous Oseledets splitting.
\end{maintheorem*}

\subsection{Overview}

An outline of the paper is as follows. In Section \ref{sec:grass}
we prove some basic results concerning Grassmanians. In Section
\ref{sec:thieullen} we describe the result of Thieullen
\cite{Thieullen} and introduce the key notions of
$\mathbb{P}$-continuity and index of compactness from that work. We
then prove our main result.  Section \ref{sec:apps} describes two
applications of our main result:  Perron--Frobenius cocycles
generated by random ``Rychlik'' maps (generalisations of
Lasota--Yorke maps), and transfer operator cocycles generated by
subshifts of finite type with random weight functions.

\section{The Grassmannian of a Banach Space}\label{sec:grass}

Let $X$ be a Banach space and suppose that $E,F\subset X$ are subspaces
forming a direct (algebraic) sum: that is, $E+F=X$ and $E\cap F=\{0\}$.
This decomposition specifies a linear map $\Proj{F}{E}(e+f)=f$ with range
$F$ and kernel $E$, called the \emph{projection onto $F$ along $E$}.
Conversely, any \emph{projection} $P:X\to X$ (that is, a linear map
satisfying $P^2=P$) determines a decomposition $X=\ker(P)+ \ran(P)$,
where $\ker(P)\cap\ran(P)=\{0\}$.

Unlike in finite dimensions, not all projections in infinite-dimensional
Banach spaces are continuous.
A necessary (and sufficient) condition for a projection to be continuous
is that it has a closed range.
Since every continuous linear map has a closed
kernel, it follows that every \emph{continuous} projection determines a
\emph{topological direct sum}: a direct sum decomposition into
complementary \emph{closed} subspaces. We denote the topological direct
sum of subspaces $E,F\subset X$ by $E\oplus F$. Conversely, it follows
from the Closed Graph Theorem that if $E\subset X$ is a closed subspace
with a closed complementary subspace $F\subset X$, then $\Proj{F}{E}$ is
continuous.

As mentioned before the
Grassmannian of $X$, denoted $\mathcal{G}(X)$, is the collection of
complemented closed subspaces of $X$.
The set $\mathcal{G}(X)$
admits a Banach manifold structure as follows. Given
$E_0\in\mathcal{G}(X)$, fix any $F_0\in\mathcal{G}(X)$ for which
$E_0\oplus F_0=X$. We can use $F_0$ to define a neighbourhood of $E_0$:
we set $U_{F_0} = \{E\in\mathcal{G}(X): E\oplus F_0 = X\}$. We then
define an isomorphism $\phi_{E_0,F_0}$ from $U_{F_0}$ to the Banach space
$L(E_0,F_0)$ by $\phi_{E_0,F_0}(E) = \Proj{F_0}{E}|E_0$. The triples
$\{U_{F_0}, \phi_{E_0,F_0}, L(E_0,F_0)\}$ form an atlas for
$\mathcal{G}(X)$ showing that near $E_0$, $\mathcal{G}(X)$ is locally
modelled on $L(E_0,F_0)$. The $(E_0,F_0)$-\emph{local norm} on  $U_{F_0}$
is defined by
\begin{equation}\label{eqn:locnorm}
\| E \|_{(E_0, F_0)} := \|\Proj{F_0}{E}|_{E_0}\|.
\end{equation}

We now prove some basic properties of the Grassmannian $\mathcal{G}(X)$.

\begin{lemma}\label{lem:projcts}
Let $X$ be a Banach space and let $\Omega$ be a topological space.
Suppose that for each $\omega\in\Omega$ there are closed subspaces
$V(\omega)$ and $W(\omega)$ whose topological direct sum is $X$. Suppose
further that $V(\omega)$ and $W(\omega)$ depend continuously on $\omega$.

Let $R(\omega)=\Proj{V(\omega)}{W(\omega)}$ be the projection of $X$ onto
$V(\omega)$ along $W(\omega)$. Then the mapping $\omega\mapsto R(\omega)$
is continuous (with respect to the operator norm on $L(X,X)$).
\end{lemma}
\begin{proof}
Let $\omega_0\in\Omega$. Since $V(\omega)$ and $W(\omega)$ are continuous
families of subspaces, there exists a neighbourhood $N_1$ of $\omega_0$
such that for all $\omega\in N_1$, $V(\omega)\oplus W(\omega_0)=X$ and
$V(\omega_0)\oplus W(\omega)=X$.

Since $X=V(\omega_0)\oplus W(\omega_0)$, both
$\|\Proj{V(\omega_0)}{W(\omega_0)}\|$ and
$\|\Proj{W(\omega_0)}{V(\omega_0)}\|$ are finite. Let $C$ be the greater
of the two.

Given any $\epsilon>0$, since $V(\omega)$ and $W(\omega)$ are continuous
there is a neighbourhood $N_2$ of $\omega_0$ contained in $N_1$ such that
for $\omega\in N_2$ one has
\begin{align*}
  &\|\Proj{W(\omega_0)}{V(\omega)}\vert_{V(\omega_0)}\|<\epsilon
  \text{; and }\\
  &\|\Proj{V(\omega_0)}{W(\omega)}\vert_{W(\omega_0)}\|<\epsilon.
\end{align*}
Let $x$ be in $X$. We now have
$x=\Proj{V(\omega)}{W(\omega)}(x)+\Proj{W(\omega)}{V(\omega)}(x)$. Write
the right side as $x_1+x_2$. Now we split $x_1$ and $x_2$ into parts
lying in $V(\omega_0)$ and $W(\omega_0)$ as $x_1=x_{11}+x_{12}$ and
$x_2=x_{21}+x_{22}$ so that
$x=x_{11}+x_{12}+x_{21}+x_{22}=(x_{11}+x_{21})+(x_{12}+x_{22})$. We have
$\Proj{V(\omega)}{W(\omega)}(x)=x_1=x_{11}+x_{12}$ and
$\Proj{V(\omega_0)}{W(\omega_0)}(x)=x_{11}+x_{21}$ so that the difference
is $x_{12}-x_{21}$.

Rearranging we have $x_{22}=x_2-x_{21}$ so that
$-x_{21}=\Proj{V(\omega_0)}{W(\omega)}(x_{22})$ so that
$\|x_{21}\|<\epsilon\|x_{22}\|$. Similarly
$\|x_{12}\|<\epsilon\|x_{11}\|$.

We have $\|x_{11}+x_{21}\|=\|\Proj{V(\omega_0)}{W(\omega_0)}(x)\|\le
C\|x\|$ so that $\|x_{11}\|\le C\|x\|+\|x_{21}\|<C\|x\|+
\epsilon\|x_{22}\|$ and similarly $\|x_{22}\|<C\|x\|+\epsilon
\|x_{11}\|$. Summing and rearranging we obtain
$\|x_{11}\|+\|x_{22}\|<2C/(1-\epsilon)\|x\|$. From the previous equations
we obtain $\|R(\omega)(x)-R(\omega_0)(x)\|\le
\|x_{12}\|+\|x_{21}\|<2C\epsilon/(1-\epsilon)\|x\|$. Since $\epsilon$ may
be chosen arbitrarily small, this establishes continuity of $R$ at
$\omega_0$.
\end{proof}

\begin{lemma}\label{lem:restnormcts}
Let $R\colon\Omega\to L(X,X)$ and $E\colon\Omega\to\mathcal G(X)$ be
continuous. Then $\omega\mapsto \|R(\omega)\vert_{E(\omega)}\|$ is
continuous.
\end{lemma}
\begin{proof}
Let $\epsilon>0$ and let $\omega_0\in\Omega$. Choose a $\delta>0$ such
that $\delta+2\delta\|R(\omega_0)\|/(1-\delta)<\epsilon$ and
$2\delta\|R(\omega_0)\|+\delta<\epsilon$.

Fix an $F_0$ such that $E(\omega_0)\oplus F_0=X$. Choose a
neighbourhood $N$ of $\omega_0$ such that for $\omega\in N$,
$E(\omega)\oplus F_0=X$,
$\|\Proj{F_0}{E(\omega)}\vert_{E(\omega_0)}\|<\delta$ and
$\|R(\omega)-R(\omega_0)\|<\delta$.

Let $x\in \overline{B}_{E(\omega)}$, the closed unit ball of $E(\omega)$.
Since $E(\omega_0)\oplus F_0=X$, $x$ may be expressed uniquely as $a+b$
with $a\in E(\omega_0)$ and $b\in F_0$. We have $a=x-b$ so that
$\Proj{F_0}{E(\omega)}(a)=-b$ yielding $\|b\|<\delta\|a\|$. We have
$\|a\|\le \|x\|+\|b\|$ so that $\|a\|<1/(1-\delta)$.

We now have
\begin{align*}
  &\|R(\omega)x\|-\|R(\omega_0)((1-\delta)a)\|\\
  &\le
  \|R(\omega)x-R(\omega_0)((1-\delta)a)\|\\
  &\le \|R(\omega)x-R(\omega_0)x\|+\|R(\omega_0)(x-(1-\delta)a)\|\\
  &\le \|R(\omega)-R(\omega_0)\|\cdot\|x\|+\|R(\omega_0)\|\cdot\|b+\delta
  a\|\\
  &\le \delta+2\delta\|R(\omega_0)\|/(1-\delta)\le \epsilon.
\end{align*}
It follows that
$\|R(\omega)\vert_{E(\omega)}\|<\|R(\omega_0)\vert_{E(\omega_0)}\|+\epsilon$.

Conversely let $x\in \overline{B}_{E(\omega_0)}$. We have
$x=\Proj{F_0}{E(\omega)}x+\Proj{E(\omega)}{F_0}x$, which we write as
$c+d$. By assumption $\|c\|<\delta$ so that $\|d\|<1+\delta$. We have
\begin{align*}
  &\|R(\omega_0)x\|-\|R(\omega)d/(1+\delta)\|\\
  &\le \|R(\omega_0)x-R(\omega)d/(1+\delta)\|\\
  &\le \|R(\omega_0)x-R(\omega_0)d/(1+\delta)\|+
  \|R(\omega_0)d/(1+\delta)-R(\omega)d/(1+\delta)\|\\
  &\le \|R(\omega_0)\|\cdot\|x-d+\delta d/(1+\delta)\|+
  \|R(\omega_0)-R(\omega)\|\cdot\|d/(1+\delta)\|\\
  &\le 2\delta\|R(\omega_0)\|+\delta<\epsilon.
\end{align*}
It follows that
$\|R(\omega_0)\vert_{E(\omega_0)}\|<\|R(\omega)\vert_{E(\omega)}\|+\epsilon$,
which establishes the required continuity.
\end{proof}

\begin{lemma}\label{lem:sumcts}
Suppose that the map $V\colon\Omega\to \mathcal G(X)$ is continuous
and that there are elements $E_0$ and $F_0$ of the Grassmannian such
that $V(\omega_0)\oplus E_0\oplus F_0=X$. Then there is a
neighbourhood $N$ of $\omega_0$ such that on $N$, $\omega\mapsto
V(\omega)\oplus F_0$ is continuous.
\end{lemma}
\begin{proof}
Since $E_0\oplus F_0$ is a topological complementary subspace of
$V(\omega_0)$, by continuity there is a neighbourhood $N_1$ of $\omega$
such that for $\omega\in N_1$, $V(\omega)\oplus E_0\oplus F_0=X$. In
particular we see that $E_0$ is a topological complementary subspace to
$V(\omega)\oplus F_0$ for $\omega\in N_1$. Let $\omega_1\in N_1$ be
fixed. We need to establish that for $\omega$ sufficiently close to
$\omega_1$, $\|\Proj{E_0}{V(\omega)\oplus F_0}\vert_{V(\omega_1)\oplus
F_0}\|$ is small. We demonstrate this by writing the operator as the
composition of three parts: two of them bounded and the third one small.

Let $\epsilon>0$. We note that $\Proj{E_0}{V(\omega)\oplus F_0}
\vert_{F_0}$ is zero so that we can rewrite $\Proj{E_0}{V(\omega)\oplus
F_0}\vert_{V(\omega_1)\oplus F_0}$ as $\Proj{E_0}{V(\omega)\oplus
F_0}\vert_{V(\omega_1)}\circ \Proj{V(\omega_1)}{F_0}$. We further
decompose $\Proj{E_0}{V(\omega)\oplus F_0}$ as $\Proj{E_0}{F_0}\circ
\Proj{E_0\oplus F_0}{V(\omega)}$ so that
\begin{equation*}
\Proj{E_0}{V(\omega)\oplus F_0}\vert_{V(\omega_1)\oplus
F_0}=\Proj{E_0}{F_0}\circ \Proj{E_0\oplus
F_0}{V(\omega)}\vert_{V(\omega_1)}\circ\Proj{V(\omega_1)}{F_0}.
\end{equation*}
Since $V(\omega_1)\oplus F_0$ and $E_0\oplus F_0$ are topological direct
sums it follows that $C=\|\Proj{V(\omega_1)}{F_0}\|$ and
$C'=\|\Proj{E_0}{F_0}\|$ are finite. By continuity of $V(\omega)$ there
is a neighbourhood $N_2$ of $\omega_1$ on which $\|\Proj{E_0\oplus
F_0}{V(\omega)} \vert_{V(\omega_1)}\|<\epsilon/(CC')$. Multiplying the
norms we see that for $\omega\in N_2$, $\|\Proj{E_0}{V(\omega)\oplus
F_0}\|<\epsilon$ as required.
\end{proof}

\begin{lemma}\label{lem:sumcts2}
Let the maps $V\colon\Omega\to \mathcal G_d(X)$ and $W:\Omega\to \mathcal
G_{d'}(X)$ be continuous and suppose that $V(\omega)\cap W(\omega)=\{0\}$
for each $\omega\in\Omega$. Then the map $\omega\mapsto V(\omega)\oplus
W(\omega)$ is continuous.
\end{lemma}

\begin{proof}
Fix $\omega_0\in\Omega$. Let $F_0$ be a topological complement of
$V(\omega_0)\oplus W(\omega_0)$ (this exists as all finite-dimensional
subspaces have a topological complement). In order to demonstrate
continuity we need to show that $\Proj{F_0}{V(\omega)\oplus
W(\omega)}\vert_{V(\omega_0)\oplus W(\omega_0)}$ has small norm. Since
$V(\omega_0)\oplus W(\omega_0)$ is a topological direct sum it is
sufficient to show that $\Proj{F_0}{V(\omega)\oplus
W(\omega)}\vert_{V(\omega_0)}$ is of small norm with a similar result for
the restriction to $W(\omega_0)$. We write
\begin{equation*}
\Proj{F_0}{V(\omega)\oplus W(\omega)}\vert_{V(\omega_0)}=
\Proj{F_0}{W(\omega)}\circ\Proj{F_0\oplus
W(\omega)}{V(\omega)}\vert_{V(\omega_0)}.
\end{equation*}

We have from Lemmas \ref{lem:projcts} and \ref{lem:sumcts} that
$\omega\mapsto \Proj{F_0\oplus W(\omega)}{V(\omega)}$ is continuous in a
neighbourhood of $\omega_0$ and from Lemma \ref{lem:restnormcts} that
$\omega\mapsto \|\Proj{F_0+W(\omega)}{V(\omega)}\vert_{V(\omega_0)}\|$ is
continuous on this neighbourhood. Since $\|\Proj{F_0+W(\omega)}{V(\omega)}\vert_{V(\omega_0)}\|=0$ when
$\omega=\omega_0$, it follows that this norm is arbitrarily small for
$\omega$ in a neighbourhood of $\omega_0$.

It remains to show that $\|\Proj{F_0}{W(\omega)}\|$ remains bounded on a
neighbourhood of $\omega_0$. To see this we note from Lemma
\ref{lem:projcts} that $\Proj{F_0\oplus V(\omega_0)}{W(\omega)}$ is
continuous on a neighbourhood of $\omega_0$ and $\Proj{F_0}{V(\omega_0)}$
is a bounded operator since $F_0\oplus V(\omega_0)$ is a topological
direct sum. Composing these two operators gives the required result.
\end{proof}

\begin{lemma}\label{lem:ctsbasis}
Let $X$ be a Banach space, $K$ a compact metrizable space and let $E:K\to
\mathcal{G}_d(X)$ be a continuous map. Let $\mathbb P$ be a finite
measure on $K$. Then there exists an open and dense measurable subset $U$
of $K$ with full $\mathbb{P}$-measure and maps $e_1,\ldots,e_d:K\to X$
with $e_i|U$ continuous, $i=1,\ldots,d$ such that for each $\omega\in U$,
$e_1(\omega),\ldots,e_d(\omega)$ is a basis for $E(\omega)$.

Furthermore, the basis can be chosen so that for each $\omega\in U$ and
all $a\in \mathbb{R}^d$,
\begin{align*}
\|a\|_2 \leq \left\|\sum_{i=1}^d a_i e_i(\omega)\right\| \leq 4 \sqrt{d}
\|a\|_2.
\end{align*}
\end{lemma}
\begin{proof}
Given $\omega_0\in K$, there exists $F_{\omega_0}\in\mathcal{G}^d(X)$
such that $E(\omega_0)\oplus F_{\omega_0} = X$. By continuity of
$E(\omega)$, there exists an open neighourhood $U_{\omega_0}$ of
$\omega_0$ such that $E(\omega)\oplus F_{\omega_0}=X$ for all $\omega\in
U_{\omega_0}$. By a theorem of F.~John (see
\cite[Chapter 4 Theorem 15]{Bollobas} for example), there exists a basis
$v_1,\ldots,v_d$ for $E(\omega_0)$ satisfying
\begin{align}
2\|a\|_2 \leq \left\|\sum_{i=1}^d a_i v_i \right\| \leq 2 \sqrt{d} \|a\|_2,
\end{align}
for all $a\in\mathbb{R}^d$, where $\|a\|_2=(\sum_{i=1}^d a_i^2)^{1/2}$ is
the Euclidean norm on $\mathbb{R}^d$. Define
$e^{\omega_0}_i:U_{\omega_0}\to X$ for each $i=1,\ldots,d$, by setting
$e^{\omega_0}_i(\omega) = \Proj{E(\omega)}{F_{\omega_0}}v_i$. Notice that
these vectors depend continuously on $\omega$ by Lemma \ref{lem:projcts}.
Replacing $U_{\omega_0}$ by a smaller neighbourhood of $\omega_0$ if
necessary, we may assume that for all $\omega\in U_{\omega_0}$ and
$a\in\mathbb{R}^d$,
\begin{align*}
\|a\|_2 \leq \left\|\sum_{i=1}^d a_i e^{\omega_0}_i(\omega)\right\| \leq
4 \sqrt{d} \|a\|_2.
\end{align*}
It follows that the vectors are linearly independent and hence form a
basis for $E(\omega)$.

We have that $\{U_\omega:\omega\in K\}$ is an open cover of $K$. Let
$\rho$ be a metric on $K$ compatible with the topology and let $\delta>0$
be the Lebesgue number of the cover: that is, for every $0<r<\delta$ and
$\omega\in K$, there exists $\omega'\in\Omega$ such that $B_r(\omega)$,
the open ball of radius $r$ centred at $\omega$, is contained in
$U_{\omega'}$. Fix $0<r_0<\delta$ and consider the open cover
$\{B_{r_0}(\omega):\omega\in K\}$ of $K$. By compactness, we have a
finite subcover $\{B_{r_0}(\omega_i):i=1,\ldots,k\}$. For each
$i=1,\ldots,k$, the collection $\{\partial B_r(\omega_i):r_0<r<\delta \}$
is an uncountable family of pairwise disjoint sets (contained in the
sphere of radius $r$ about $\omega_i$), and so there exists
$r\in(r_0,\delta)$ such that $\mathbb{P}(\partial B_{r}(\omega_i))=0$ for
each $i$.

We have that $\{B_i:=B_r(\omega_i):i=1,\ldots,k\}$ is a cover of $K$ by
open sets whose boundaries have zero $\mathbb{P}$-measure. These sets
have the additional property that for each $i=1,\ldots,k$, there exists
$\omega_i'\in K$ such that $B_i\subset U_{\omega_i'}$. Set
$D_i=B_i\backslash \bigcup_{j<i}B_j$ and let
$U=\bigcup_{i=1}^k\text{int}(D_i)$. Since $K\backslash
U\subset\bigcup_{i=1}^k\partial B_i$, $U$ is an open dense set of full
$\mathbb{P}$-measure and $\{D_i:i=1,\ldots,k\}$ is a partition of $K$.
Setting $e_i(\omega) = e^{\omega_j'}_i(\omega)$ for $\omega\in D_j$, for
each $i=1,\ldots,d$ and $j=1,\ldots,k$ gives maps with the required
properties.
\end{proof}

\section{Oseledets splitting}\label{sec:thieullen}

Thieullen \cite{Thieullen} in his work on multiplicative ergodic
theorems for operators introduced a framework on which this paper
will be based. A key notion introduced in that paper is $\mathbb
P$-continuity.

\begin{definition} \label{def:Pcont} For a topological space
$\Omega$, equipped with a Borel probability $\mathbb{P}$, a mapping
$f$ from $\Omega$ to a topological space $Y$ is said to be
$\mathbb P$-\textsl{continuous} if $\Omega$ can be expressed as a
countable union of Borel sets such that the restriction of $f$
to each is continuous.
\end{definition}

\begin{remark}
\label{rem:Pcont}
As noted in \cite{Thieullen},
if $\Omega$ is homeomorphic to a Borel subset of a separable complete
metric space, then a function $f:\Omega\to Y$ is $\mathbb
P$-continuous if and only if there exists a sequence $(K_n)_{n\geq 0}$ of
pairwise disjoint compact subsets of $X$ such that $\mu(\bigcup_{n\geq 0}
K_n)=1$ and the restriction $f|_{K_n}$ is continuous for each $n\geq 0$.
\end{remark}

We shall call a Lyapunov filtration or Oseledets splitting
$\mathbb{P}$-continuous if all of the exponents and all maps into the
Grassmannian are $\mathbb{P}$-continuous (with respect to the topology
defined in Section \ref{sec:grass} in the case of maps into the
Grassmannian).

\begin{remark}
If $\mathbb{P}$ is a \emph{Radon} measure on $\Omega$ (that is, locally
finite and tight) and $Y$ is a metric space, then a map $f:\Omega\to Y$
is $\mathbb{P}$-continuous if and only if it is measurable (see
\cite{Fremlin}). In particular, this is the case in the `Polish noise'
setting (see, for example, \cite{KiferLiu,Arbieto}), where $\Omega$ is a
separable topological space with a complete metric, $\mathcal F$ is the
Borel sigma-algebra and $\mathbb P$ is any Borel probability.
\end{remark}

Consider a random dynamical system $\mathcal
R=(\Omega,\mathcal{F},\mathbb{P},\sigma,X,\mathcal L)$. If $\sigma$ is
invertible with a measurable inverse, we say $\mathcal{R}$ has an
\emph{invertible base}. If $\Omega$ is a Borel subset of a complete
separable metric space, $\mathcal{F}$ is the Borel sigma-algebra and
$\sigma$ is continuous (or a homeomorphism), we say $\mathcal{R}$ has a
\emph{continuous (or homeomorphic) base}.

Suppose $\mathcal R$ is a random dynamical system with a homeomorphic base.
Provided $\omega\mapsto \mathcal L_\omega$ is $\mathbb{P}$-continuous we
see that $\omega\mapsto\|\mathcal L^{(n)}_\omega\|$ is
$\mathbb{P}$-continuous and hence $\mathcal F$-measurable. We shall assume
throughout that $\int \log^+ \| \mathcal
L_\omega\|\,\mathrm{d}\mathbb{P}(\omega)<\infty$. Since $\log\|\mathcal
L^{(n)}_\omega\|$ is a subadditive sequence of functions it follows
from the subadditive ergodic theorem that for almost every $\omega$,
$\frac1n\log\|\mathcal L^{(n)}_\omega\|$ is convergent and hence the
quantity $\lambda(\omega)$ defined in \eqref{eqn:lambda} may be
re-expressed as
\begin{equation*}
  \lambda(\omega)=\lim_{n\to\infty}(1/n)\log \|\mathcal L^{(n)}_\omega\|.
\end{equation*}
The boundedness of $\int\log^+\|\mathcal
L_\omega\|\,\mathrm{d}\mathbb{P}(\omega)$ ensures that $\lambda(\omega)$
is finite $\mathbb P$-almost everywhere.

\begin{proposition}\label{prop:topexponent}
For each $\omega\in\Omega$, we have $\sup \Lambda(\omega) =
\lambda(\omega)$.
\end{proposition}
\begin{proof}
Clearly $\sup \Lambda(\omega) \leq \lambda(\omega)$, so we show that
$\sup \Lambda(\omega) \geq \lambda(\omega)$. Fix $\omega\in\Omega$, let
$r> \sup_{v\in X} \lambda(\omega,v)$. Set $A_N=\{v\in X: \|\mathcal
L^{(n)}_\omega v\| \leq N e^{nr},\forall n\in\mathbb{N}\}$. The set $A_N$
is closed, and by the choice of $r$, we have $\bigcup_{N\in \mathbb{N}}
A_N = X$ for each $\omega\in\Omega$. Thus by the Baire Category Theorem,
there exists an $A_N$ containing an interior point $u$. Let $\delta>0$ be
small enough that $B_\delta(u)\subset A_N$. For any $v\in
B_\delta(0)$ and $n>0$, we have $\|\mathcal L^{(n)}_\omega
v\|\leq \|\mathcal L^{(n)}_\omega(v-u)\| +\|\mathcal L^{(n)}_\omega u\|$.
So $\|\mathcal L^{(n)}_\omega\|\leq (2N/\delta) e^{nr}$, and hence
$\lambda(\omega) \leq r$. Since $r$ is an arbitrary quantity greater than
$\Lambda(\omega)$, the result follows.
\end{proof}

We concentrate on the setting
in which $\sigma$ is ergodic. The function $\lambda(\omega)$ is then
constant along orbits, and thus essentially constant. We denote by
$\lambda^*\in\mathbb{R}$ the constant satisfying
$\lambda(\omega)=\lambda^*$ for almost every $\omega\in\Omega$.

A second key concept introduced by Thieullen is that of the index of
compactness of a random composition of operators.
For a bounded operator $A$, $\|A\|_\text{ic}$ is defined to be the
infimal $r$ such that $A(B_X)$ may be covered by a finite number
of $r$-balls, where $B_X$ is the unit ball in $X$.
We have $\|AA'\|_\text{ic}\le
\|A\|_\text{ic}\|A'\|_\text{ic}$ for any bounded linear operators on
$X$. One can check that $|\|A\|_\text{ic}-\|A'\|_\text{ic}|\le
\|A-A'\|$ so that $\|A\|_\text{ic}$ is a continuous function of the
operator. In particular for each $n$, $\omega\mapsto \|\mathcal
L^n_\omega\|_\text{ic}$ is $\mathbb{P}$-continuous and hence $\mathcal
F$-measurable. By sub-additivity we have $(1/n)\log \|\mathcal
L^{(n)}_\omega\|_\text{ic}$ is convergent.

\begin{definition}
\label{def:ic} The limit $\kappa(\omega):=\lim_{n\to\infty}(1/n)\log
\|\mathcal L^{(n)}_\omega\|_\text{ic}$ is called the \textsl{index
of compactness} of the random composition of operators.
\end{definition}
 Since $\kappa(\omega)$ is $\sigma$-invariant it is
equal almost everywhere to a constant which we call $\kappa^*$.

\begin{theorem}[Thieullen \cite{Thieullen}]\label{thm:Thieullen}
Let $\mathcal{R}=(\Omega,\mathcal{F},\mathbb{P},\sigma,X,\mathcal L)$ be
a random dynamical system with an ergodic continuous base, and suppose
$\omega\mapsto\mathcal L_\omega$ is $\mathbb{P}$-continuous, and that
$\int \log^+\|\mathcal L_\omega\|\,\mathrm{d}\mathbb{P} <+\infty$. If
$\kappa^*<\lambda^*$, then $\mathcal{R}$ is quasi-compact, with
$\alpha(\omega) =\kappa^*$ a.\,e., and admits a $\mathbb{P}$-continuous
Lyapunov filtration.

Moreover, if the base is invertible and $\mathcal L_\omega$ is injective
a.\,e., then $\mathcal{R}$ admits a $\mathbb{P}$-continuous Oseledets
splitting.
\end{theorem}

Our main result in this article 
is the extension of
Thieullen's theorem to show that one obtains an Oseledets splitting in
Thieullen's setting without making the assumption of invertibility of the
$\mathcal L_\omega$.

\begin{theorem}\label{thm:main}
Let $\Omega$ be a Borel subset of a separable complete metric space,
$\mathcal{F}$ the Borel sigma-algebra and $\mathbb{P}$ a Borel
probabilty. Let $X$ be a Banach space and consider a random dynamical
system $\mathcal{R}=(\Omega,\mathcal{F},\mathbb{P},\sigma,X,\mathcal L)$
with base transformation $\sigma:\Omega\to\Omega$ an ergodic
homeomorphism, and suppose that the generator $\mathcal L:\Omega\to
L(X,X)$ is $\mathbb{P}$-continuous and satisfies
$$
\int \log^+\|\mathcal L_\omega\|\,\mathrm{d}\mathbb{P} <+\infty.
$$
If $\kappa^*<\lambda^*$ for almost every $\omega$, then $\mathcal{R}$ is
quasi-compact and admits a unique $\mathbb{P}$-continuous Oseledets
splitting.
\end{theorem}

The proof of this theorem (which makes extensive use of Theorem
\ref{thm:Thieullen}) is given in the next two subsections, in which
existence and uniqueness of the Oseledets splitting, respectively,
are proved.

\subsection{Existence of an Oseledets splitting}

Consider a random dynamical system
$\mathcal{R}=(\Omega,\mathcal{F},\mathbb{P},\sigma,X,\mathcal L)$ with an
ergodic homeomorphic base. Suppose $\mathcal L$ is
$\mathbb{P}$-continuous and $\int \log^+\|\mathcal
L_\omega\|\,\mathrm{d}\mathbb{P} <+\infty$. Let $\mathrm{EX}(\mathcal
R)=\{(\lambda_i,d_i)\}_{i=1}^p$ be the exceptional Lyapunov spectrum of
$\mathcal{R}$, and $(V_i:\Omega\to\mathcal{G}^\infty(X))_{i=1}^p$
the Lyapunov filtration.

Following Thieullen, we construct an extension Banach space $\tilde{X}$,
and a new generator $\tilde{\mathcal L}:\Omega\to L(\tilde{X},\tilde{X})$
whose cocycle retains
all the dynamical information of the original system but has the
advantage that $\tilde{\mathcal L}_\omega$ is injective.

The extended random dynamical system
$\tilde{\mathcal{R}}=(\Omega,\mathcal{F},\mathbb{P},\sigma,
\tilde{X},\tilde{\mathcal L})$ is defined as follows:
\begin{align*}
\tilde{X}=\left\{ (v_n)_{n=0}^\infty : \forall n,\  v_n\in X,\
\mathrm{sup}_n \|v_n\| <\infty  \right\}, \\
\tilde{\mathcal L}_\omega(v_0,v_1,v_2,\ldots) =  (\mathcal L_\omega v_0,
\alpha_0 v_0,\alpha_1 v_1,\alpha_2 v_2,\ldots),
\end{align*}
for a positive sequence $(\alpha_n)_{n=0}^\infty$ decaying to zero. We
endow $\tilde{X}$ with the norm $\|\tilde v\|_{\tilde X}=\sup_n
\|v_n\|_X$ where $\tilde v=(v_n)_{n=0}^\infty$. Every $\tilde{\mathcal
L}_\omega$ is injective on $\tilde{X}$. In Thieullen's article sufficient
conditions on the speed of decay of the sequence $(\alpha_n)$ are given
to ensure that the indices of compactness of $\tilde{\mathcal{R}}$ and
$\mathcal{R}$ are equal ($\tilde\kappa^*=\kappa^*$) and that
$\tilde\lambda^*=\lambda^*$. In fact we check in Subsection
\ref{subsec:Thieullensuff} that this holds for any sequence $(\alpha_n)$
of positive numbers tending to 0.

Provided $\kappa^*<\lambda^*$, we may apply the invertible form of
Thieullen's Theorem to $\tilde {\mathcal R}$ to obtain the
$\sigma$-invariant subset $\Omega'\subset \Omega$,
$\mathbb{P}(\Omega')=1$, $\mathbb{P}$-continuous Lyapunov filtration
$(\tilde{V}_i:\Omega\to\mathcal{G}^\infty(\tilde{X}))_{i=1}^p$ and
Oseledets subspaces
$(\tilde{E}_i:\Omega\to\mathcal{G}_\infty(\tilde{X}))_{i=1}^p$. We denote
by $\tilde{\lambda}(\omega,v):=\lim_{n\to\infty} \log\| \tilde{\mathcal
L}^{(n)}_\omega\tilde v \|$ the Lyapunov exponents for
$\tilde{\mathcal{R}}$.

Let $\pi:\tilde{X}\to X$ denote the (continuous) mapping onto the zeroth
coordinate. We have $\mathcal L\circ \pi = \pi\circ \tilde{\mathcal L}$.
Thieullen proves that for all $\omega\in\Omega'$ and $\tilde v\in\tilde
X$, $\lambda(\omega,\pi(\tilde v))=\tilde\lambda(\omega,\tilde v)$, and
that $\mathcal{R}$ and $\tilde{\mathcal{R}}$  have the same expectional
exponents. He then defines $V_i(\omega)=\pi(\tilde V_i(\omega))$ and
proves that the $(V_i:\Omega\to \mathcal G^\infty(X))_{i=1}^p$ form a Lyapunov filtration for the one-sided
system.

For each $1\leq i\leq p$, we define $E_i(\omega)=\pi
\tilde{E}_i(\omega)$. As the linear image of a finite dimensional
space, $E_i(\omega)$ is a closed subspace. We now demonstrate that
$(E_i:\Omega\to\mathcal{G}(X))_{i=1}^p$ is the splitting we seek.

\begin{claim}\label{clm:existence}
The maps $(E_i:\Omega\to\mathcal{G}(X))_{i=1}^p$ form a
$\mathbb{P}$-continuous Oseledets splitting for
$\mathcal{R}=(\Omega,\mathcal{F},\mathbb{P},\sigma,X,\mathcal L)$.
\end{claim}
\begin{proof}
Let $v\in V_{i+1}(\omega)\cap E_i(\omega)$. Then $v=\pi(\tilde v)$ with
$\tilde v\in \tilde V_{i+1}(\omega)$ so that $\lambda(\omega,v)=\tilde
\lambda(\omega,\tilde v)\le \lambda_{i+1}$. On the other hand if $v\ne
0$, then $v=\pi(\tilde v')$ for some $\tilde v'\in \tilde
E_i(\omega)\backslash\{0\}$. In this case we obtain
$\lambda(\omega,v)=\tilde\lambda(\omega,\tilde v')=\lambda_i$. Since this
contradicts the fact that $\lambda(\omega,v)\le\lambda_{i+1}$ it follows
that $V_{i+1}(\omega)\cap E_i(\omega)=\{0\}$. Since $\tilde
V_i(\omega)=\tilde V_{i+1}(\omega)+\tilde E_i(\omega)$, any $v\in
V_i(\omega)$ can be written as $u+w$ with $u\in V_{i+1}(\omega)$ and
$w\in E_i(\omega)$. By the triviality of the intersection this
decomposition is unique. Since both spaces are closed ($V_{i+1}(\omega)$
by Thieullen's theorem and $E_i(\omega)$ by finiteness of dimension) we
obtain $V_i(\omega)=V_{i+1}(\omega)\oplus E_i(\omega)$.

We now show that $\dim E_i(\omega) =\dim \tilde E_i(\omega)$. If $\tilde
v\in\tilde{X}$ satisfies $\pi(\tilde v)=0$, then
$\tilde{\lambda}(\omega,\tilde v)= -\infty$, and so $\ker(\pi)\cap
\tilde{E}_i(\omega)=\{0\}$. Since $E_i(\omega)=\pi(\tilde E_i(\omega))$,
we see that $\dim E_i(\omega)\le \dim \tilde E_i(\omega)$. Suppose the
subspace $\tilde{E}_i(\omega)$ is $k$-dimensional and can be written as
$\tilde{E}_i(\omega)=\langle \tilde v^1, \ldots, \tilde v^k \rangle$.
Then $\langle v^i_0 := \pi(\tilde v^i):i=1,\ldots,k\rangle =
E_i(\omega)$. If we had $c_1 v^1_0 + \cdots + c_k v^k_0 = 0$ for some
$c\in \mathbb{R}^k\backslash\{0\}$, then we would have $c_1 \tilde v^1 +
\cdots c_k \tilde v^k \in\ker(\pi)\cap \tilde{E}_i(\omega)=\{0\}$,
contradicting linear independence of the $\tilde v^1,\ldots,\tilde v^k$.
Hence, $v^1_0,\ldots,v^k_0$, are also linearly independent, and so
$d_i=\dim E_i(\omega)=\dim \tilde{E}_i(\omega)$.

We end by proving the $\mathbb{P}$-continuity of the subspaces $E_i$.
Since the function $\tilde{E}_i:\Omega\to\mathcal{G}_{d_i}(\tilde{X})$ is
$\mathbb{P}$-continuous, there exists a sequence $(K_n)_{n\geq 0}$ of
pairwise disjoint compact subsets of $\Omega$ such that
$\mathbb{P}(\bigcup_{n\geq 0 } K_n)=1$ and $\tilde{E}_i|_{K_n}$ is
continuous for each $n\geq 0$. By Lemma \ref{lem:ctsbasis}, for each
$n\geq 0$ there exists an open and dense subset $U_n\subset K_n$,
$\mathbb{P}(U_n)=\mathbb{P}(K_n)$ and continuous functions
$\tilde{e}^{i,n}_j:U_n\to \tilde{X}$, $j=1,\ldots,d_i$, with
$\tilde{e}^{i,n}_1(\omega),\ldots,\tilde{e}^{i,n}_{d_i}(\omega)$ forming
a basis for $\tilde{E}_i(\omega)$ for each $\omega\in U_n$. Since
$\pi:\tilde{X}\to X$ is continuous, the functions
$e^{i,n}_j:=\pi\circ\tilde{e}^{i,n}_j:U_n\to X$ are continuous, and
$e^{i,n}_1(\omega),\ldots,e^{i,n}_{d_i}(\omega)$ forms a basis for
$E_i(\omega)$ as shown above. Take functions $e^i_j:\Omega\to X$,
$j=1,\ldots,d_i$, satisfying $e^i_j(\omega)=e^{i,n}_1(\omega)$ for
$\omega\in U_n$, $n\geq 0$. Applying Lemma \ref{lem:sumcts2} inductively
we see that $\omega\mapsto \langle
e^{i}_1(\omega),\ldots,e^{i}_{d_i}(\omega)\rangle
\in\mathcal{G}_{d_i}(X)$ is continuous on $U_n$, for each $n\geq 0$ and
so $\mathbb{P}$-continuous on $\Omega$, which shows that
$E_i:\Omega\to\mathcal{G}_{d_i}(X)$ is $\mathbb{P}$-continuous.
\end{proof}

\subsection{Uniqueness of the Oseledets splitting}
Consider
$\mathcal{R}=(\Omega,\mathcal{F},\mathbb{P},\sigma,X,\mathcal L)$, a
quasi-compact random dynamical system, and assume that $\sigma$ is
ergodic and $\int \log^+\|\mathcal
L_\omega\|\,\mathrm{d}\mathbb{P}(\omega) <\infty$. Let
$\mathrm{EX}(\mathcal{R})=\{(\lambda_i,d_i)\}_{i=1}^p$ be the
exceptional Lyapunov spectrum,
$(V_i:\Omega\to\mathcal{G}^{c_i}(X))_{i=1}^p$ the Lyapunov
filtration and $(E_i:\Omega\to\mathcal{G}_{d_i}(X))_{i=1}^p$ the
Oseledets subspaces constructed above.

The following lemma gives us exponential uniformity in a
finite-dimensional subspace all of whose Lyapunov exponents are equal. A
result of this type first appeared in the Euclidean case in a paper of
Barreira and Silva \cite{BarreiraSilva} (see also \cite{FLQ} for an
independent proof). The proof here follows by choosing a suitable basis.

\begin{lemma}\label{lem:uniformity}
Let $B\colon \Omega\to L(X,X)$ be a $\mathbb{P}$-continuous family of
operators and let $E\colon\Omega\to \mathcal G_d(X)$ be
$\mathbb{P}$-continuous. Suppose that $B(\omega)$ maps $E(\omega)$
bijectively to $E(\sigma\omega)$. If for almost every $\omega$,
$\lim_{n\to\infty} (1/n)\log \|B^{(n)}_\omega v\|\to \lambda$ for all $v\in
E(\omega)\backslash\{0\}$ (i.e. if all Lyapunov exponents of $B$ are
equal to $\lambda$) then
\begin{equation*}
  \lim_{n\to\infty}\frac1n\log\inf_{x\in S_{E(\omega)}}\|B^{(n)}_\omega
  x\|=\lim_{n\to\infty}\frac1n\log\sup_{x\in S_{E(\omega)}}\|B^{(n)}_\omega
  x\|=\lambda,
\end{equation*}
where $S_{E(\omega)}$ denotes the unit sphere $\{x\in
E(\omega):\|x\|=1\}$.
\end{lemma}
\begin{proof}
By Lemma \ref{lem:ctsbasis}, since $E:\Omega\to\mathcal{G}_{d}(X)$ is
$\mathbb{P}$-continuous, we have $\mathbb{P}$-continuous functions
$f_i:\Omega\to X$ satisfying
\begin{equation}\label{eqn:nicebasis}
\|a\|_2 \leq \left\|\sum_{i=1}^{d} a_i f_i(\omega)\right\| \leq 4
\sqrt{{d}} \|a\|_2,
\end{equation}
where $\|\cdot\|_2$ represents the Euclidean norm on $\mathbb{R}^d$. Let
$A(\omega):\mathbb{R}^{d}\to E(\omega)$ be the map given by
$A(\omega)a=\sum_{i=1}^{d}a_i f_i(\omega)$. By \eqref{eqn:nicebasis} we
have $\|a\|_2 \leq \|A(\omega)a\|\leq 4\sqrt{d} \|a\|_2$. The linear map
$A$ is invertible a.\,e.~ and satisfies $1/(4\sqrt{d})\|v\| \leq
\|A(\omega)^{-1}v\|_2\leq \|v\|$ for $v\in E(\omega)$. We have a
cocycle $\tau$ on $\mathbb{R}^{d}$ given by $\tau^{(n)}(\omega) :=
A(\sigma^n\omega)^{-1}B^{(n)}_\omega A(\omega)$. If $a\in
\mathbb{R}^{d}$, then
\begin{align*}
\|\tau^{(n)}(\omega)a\|_2 & \leq \|A(\sigma^n\omega)^{-1}\|
\cdot \|B^{(n)}_\omega A(\omega)a\| \\
& \leq \|B^{(n)}_\omega (A(\omega)a)\| \textrm{\ and} \\
\|\tau^{(n)}(\omega)a\|_2 & \geq \|B^{(n)}_\omega (A(\omega)a)\|
/\|A(\sigma^n\omega)\| \\
& \geq \frac{1}{4\sqrt{d}}\|B^{(n)}_\omega(A(\omega)a)\|.
\end{align*}
Since $A(\omega)$ is a bijection, it follows that
$\lim_{n\to\infty}(1/n)\log\|\tau^{(n)}(\omega)a\|_2=\lambda_i$ for each
$a\in\mathbb{R}^d\backslash\{0\}$. Applying the theorem of Barreira and Silva
\cite[Theorem 2]{BarreiraSilva} (or see \cite[Proof of Theorem
4.1]{FLQ} ), we have that
\begin{align*}
&\lim_{n\to\infty} (1/n)\log \inf \{\|\tau^{(n)}(\omega)a
\|_2:a\in\mathbb{R}^{d},\|a\|_2=1\}\\
=&\lim_{n\to\infty} (1/n)\log \sup \{\|\tau^{(n)}(\omega)a
\|_2:a\in\mathbb{R}^{d},\|a\|_2=1\}=\lambda.
\end{align*}
Reusing the above inequalities the proof of the Lemma is complete.
\end{proof}

A sequence $(v_n)_{n\in\mathbb{Z}}$ is called a \emph{full orbit} at
$\omega\in\Omega$ if $\mathcal L(\sigma^n\omega)v_n=v_{n+1}$ for all
$n\in\mathbb{Z}$. For full orbits, we may consider growth rates as
$n\to-\infty$.

\begin{lemma}\label{lem:fullorbits}
Let $(v_n)_{n\in\mathbb{Z}}\subset X$ be a full orbit for
$\omega\in\Omega'$ and suppose $v_n\in V_i(\sigma^n\omega)$ for all
$n\in\mathbb{Z}$. Then
\begin{equation*}
\liminf_{n\to\infty}(1/n)\log \| v_{-n} \| \geq - \lambda_i.
\end{equation*}
If we have $v_n\in E_i(\sigma^n\omega)$ for all $n\in\mathbb{Z}$, then we
have the stronger statement
\begin{equation*}
\lim_{n\to\infty}(1/n)\log \| v_{-n} \| =- \lambda_i.
\end{equation*}
\end{lemma}
\begin{proof}
We have $\lim_{n\to\infty} (1/n)\log \|\mathcal
L^{(n)}_\omega|V_i(\omega)\| = \lambda_i$, and by \cite[Lemma 8.2]{FLQ},
it follows that $\lim_{n\to\infty} (1/n)\log \|\mathcal
L^{(n)}(\sigma^{-n}\omega)|V_i(\sigma^{-n}\omega)\| = \lambda_i$. Thus
for any full orbit $\{v_n\}_{n\in\mathbb{Z}}$ satisfying $0\neq v_n\in
V_i(\sigma^n\omega)$ for all $n\in\mathbb{Z}$, we have
\begin{align}
\limsup_{n\to\infty}(1/n)\log (\|\mathcal
L^{(n)}(\sigma^{-n}\omega)v_{-n}\|/\|v_{-n}\|) \leq \lambda_i.
\end{align}
Thus
\begin{align*}
\liminf_{n\to\infty} \frac{1}{n}\log \|v_{-n}\| &= - \limsup_{n\to\infty}
\frac{1}{n}\log \frac{\|v_0\|}{\|v_{-n}\|} \\
&= - \limsup_{n\to\infty} \frac{1}{n}\log \frac{\|\mathcal
L^{(n)}(\sigma^{-n}\omega)v_{-n}\|}{\|v_{-n}\|} \geq -\lambda_i.
\end{align*}

For the second statement we shall assume that $v_n\in
E_i(\sigma^n\omega)$ for all $n\in\mathbb{Z}$. The mapping $\mathcal
L_\omega|E_i(\omega)$ is a bijection, so we denote by
$S(\omega):E_i(\sigma\omega)\to E_i(\omega)$ the inverse map. We let
$S^{(n)}(\omega):= S(\sigma^{-n}\omega)\cdots
S(\sigma^{-1}\omega)=[\mathcal L^{(n)}_{\sigma^{-n}\omega}|
_{E_i(\sigma^{-n}\omega)}]^{-1}$ denote the cocycle for the map
$\sigma^{-1}$ generated by $S$. As $\log\|S^{(n)}(\omega)\|$ is a
subadditive sequence of functions over $\sigma^{-1}$, using \cite[Lemma 8.2]{FLQ}
again we have
\begin{align*}
\lim_{n\to\infty}\frac{1}{n}\log\|S^{(n)}(\omega)\| & =
\lim_{n\to\infty}\frac{1}{n}\log\|S^{(n)}(\sigma^n\omega)\| \\
& = - \lim_{n\to\infty}\frac{1}{n}\log\inf_{0\neq v\in
E_i(\omega)}\frac{\|\mathcal L^{(n)}_\omega v\|}{\|v\|} \\
& = -\lambda_i,
\end{align*}
where the last equality follows from Lemma \ref{lem:uniformity}. Suppose
now that $0\neq v_n\in E_i(\sigma^n\omega)$ for all $n\in\mathbb{Z}$.
Then we have
\begin{equation*}
\limsup_{n\to\infty}\frac{1}{n}\log\|v_{-n}\| =
\limsup_{n\to\infty}\frac{1}{n}\log\frac{\|S^{(n)}(\omega)v_0\|}{\|v_0\|}
\leq -\lambda_i.
\end{equation*}
\end{proof}

\begin{claim}\label{clm:uniqueness}
The $\mathbb{P}$-continuous Oseledets splitting is unique on a full
measure subset of $\Omega$.
\end{claim}
\begin{proof}
Fix $1\leq i \leq p$. Consider a $\mathbb{P}$-continuous map
$E_i':\Omega\to\mathcal{G}_{d_i}(X)$ satisfying $\mathcal L_\omega
E_i'(\omega) = E_i'(\sigma\omega)$ and $E_i'(\omega)\oplus
V_{i+1}(\omega) = V_i(\omega)$ for almost every $\omega\in\Omega$. Assume
for a contradiction that there is a measurable subset $J\subset \Omega$,
$\mathbb{P}(J)>0$, such that $E_i(\omega)\neq E'_i(\omega)$ for all
$\omega\in J$.

Let $F_i(\omega)=\bigoplus_{j<i} E_j(\omega)$. We have $V_{i+1}(\omega)
\oplus E_i(\omega) \oplus F_i(\omega)=X$ for all $\omega\in\Omega'$. Let
$(U_n)_{n\geq 0}$ be a sequence of measurable subsets of $\Omega$,
$\mathbb{P}(\bigcup_{n\geq 0} U_n)=1$, such that the maps
$V_{i+1}|_{U_n}$, $E_{i}|_{U_n}$ and $F_{i}|_{U_n}$ are continuous. By
Lemma \ref{lem:sumcts2}, the map $E_i\oplus F_i$ is continuous on $U_n$
for each $n\geq 0$. By Lemma \ref{lem:projcts}, the map
$R(\omega):=\Proj{V_{i+1}(\omega)}{E_i(\omega)\oplus F_i(\omega)}$ is
continuous on $U_n$ for each $n\geq 0$. Thus, by Lemma
\ref{lem:restnormcts}, the mapping
$g(\omega)=\|R(\omega)\vert_{E_i'(\omega)}\|$ is $\mathbb{P}$-continuous,
and in particular, is $\mathcal F$-measurable.

We first prove that $\lim_{n\to\infty}g(\sigma^n\omega)=0$ for almost all
$\omega$. Let $\omega\in\Omega'$ be given. For any fixed $u\in
E_i'(\omega)\setminus\{0\}$, we have $R(\omega)u\in V_{i+1}(\omega)$ so
that for any $\epsilon>0$ there exists a $C<\infty$ with $\|\mathcal
L^{(n)}_\omega R(\omega)u\|\le Ce^{n(\lambda_{i+1}+\epsilon)}$ for all
$n>0$.

On the other hand since $u\in V_i(\omega)\setminus V_{i+1}(\omega)$,
there is a $C'>0$ such that $\|\mathcal L^{(n)}_\omega u\|\ge
C'e^{n(\lambda_i-\epsilon)}$ for all $n$. Fix
$\epsilon<\frac14(\lambda_i-\lambda_{i-1})$. We have for each fixed $u$
there is a constant $C_u$ such that
\begin{equation*}
  \frac{\|\mathcal L^{(n)}_\omega R(\omega)u\|}{\|\mathcal L^{(n)}_\omega
  u\|}\le
  C_u e^{-n(\lambda_i-\lambda_{i+1}-2\epsilon)}\text{  for all $n>0$.}
\end{equation*}
We now use a Baire category argument. Define $D_N$ by
\begin{equation*}
D_N=\{u\in E'_i(\omega)\colon \|\mathcal L^{(n)}_\omega R(\omega)u\|\le
Ne^{-n(\lambda_i-\lambda_{i+1}-2\epsilon)}\|\mathcal L^{(n)}_\omega
u\|\,\forall n>0\}.
\end{equation*}
Since these sets are closed and their union is all of $E'_i(\omega)$, one
of them must contain a ball $\overline{B_\delta(u)}\cap E'_i(\omega)$. By
scale-invariance it contains a ball $\overline{B_1(u/\delta)}\cap
E'_i(\omega)$. Set $u_0=u/\delta$ and let $x\in E'_i(\omega)$ satisfy
$\|x\|=1$. Then we have for each $n$
\begin{align*}
  \|\mathcal L^{(n)}_\omega R(\omega)(u_0+x)\|&\le
Ne^{-n(\lambda_i-\lambda_{i+1}-2\epsilon)}\|\mathcal
L^{(n)}_\omega(u_0+x)\|
\\
\|\mathcal L^{(n)}_\omega R(\omega)u_0\|&\le
Ne^{-n(\lambda_i-\lambda_{i+1}-2\epsilon)}\|\mathcal L^{(n)}_\omega
u_0\|.
\end{align*}
Using $\mathcal L^{(n)}_\omega R(\omega)=R(\sigma^n\omega)\mathcal
L^{(n)}_\omega$, subtracting the above two inequalities and using the
triangle inequality we obtain
\begin{equation*}
  \|R(\sigma^n\omega)\mathcal L^{(n)}_\omega x\|\le
  Ne^{-n(\lambda_i-\lambda_{i+1}-2\epsilon)}
  (\|\mathcal L^{(n)}_\omega(u_0+x)\|+\|\mathcal L^{(n)}_\omega u_0\|).
\end{equation*}

Since $\mathcal L^{(n)}_\omega x/\|\mathcal L^{(n)}_\omega x\|$ is a
general point of the intersection of the unit sphere with
$E'_i(\sigma^n\omega)$ we obtain
\begin{equation*}
  g(\sigma^n\omega)\le Ne^{-n(\lambda_i-\lambda_{i+1}-2\epsilon)}
  \frac{\sup_{x\in S_{E'_i(\omega)}} \|\mathcal L^{(n)}_\omega(u_0+x)\|
  +\|\mathcal L^{(n)}_\omega u_0\|}
  {\inf_{x\in S_{E'_i(\omega)}}\|\mathcal L^{(n)}_\omega x\|},
\end{equation*}
The numerator is bounded above by an expression of the form
$Ce^{n(\lambda_i+\epsilon)}$. Similarly, by Lemma \ref{lem:uniformity}
the denominator is bounded below by an expression of the form
$C'e^{n(\lambda_i-\epsilon)}$. It follows that $g(\sigma^n\omega)\le
(NC/C')e^{-n(\lambda_i-\lambda_{i+1}-4\epsilon)}$. By our choice of
$\epsilon$ we see that $g(\sigma^n\omega)\to 0$ as claimed.

Now let $\omega\in J$ and let $(v_n)$ be a full orbit over $\omega$ with
$v_0\in E'_i(\omega)\backslash E_i(\omega)$. Such an orbit exists since
$\mathcal L_\omega$ maps $E'_i(\omega)$ bijectively to
$E'_i(\sigma(\omega))$. Let $u_n=v_n-R(\sigma^n\omega)v_n$ and
$w_n=R(\sigma^n\omega)v_n$. Since $E'_i(\omega)\subset E_i(\omega)\oplus
V_{i+1}(\omega)$ we see that $u_n\in E_i(\sigma^n\omega)$ (i.e. $u_n$ has
no component in $F_i(\sigma^n\omega)$). We also have $w_n\in
V_{i+1}(\sigma^n\omega)$. Since $w_0\ne 0$ we have $w_n\ne 0$ for all
$n<0$.

We now have $R(\sigma^{-n}\omega)(w_{-n}+u_{-n})=w_{-n}$. Lemma
\ref{lem:fullorbits} tells us that $\|u_{-n}\|\le
Ce^{-n(\lambda_i-\epsilon)}$ and that $\|w_{-n}\|\ge
C'e^{-n(\lambda_{i+1}+\epsilon)}$. We deduce that
\begin{align*}
\liminf_{n\to\infty}g(\sigma^{-n}\omega) & =
\liminf_{n\to\infty}\|R(\sigma^{-n}\omega)
\vert_{E'_i(\sigma^{-n}\omega)}\| \\
& \ge \liminf_{n\to\infty}\|w_{-n}\|/(\|w_{-n}\|+\|u_{-n}\|) =  1.
\end{align*}
If we consider the set $A=\{\omega\in\Omega: g(\omega)<1/2\}$ we have for
almost every $\omega$, $\sigma^n \omega\in A$ for all large positive $n$
whereas $\sigma^n\omega\not\in A$ for all large negative $n$. This
contradicts the Poincar\'e recurrence theorem, and hence the promised
uniqueness is established.
\end{proof}

\subsection{Necessity of invertibility of the base}\label{sec:invbase}

The Main Theorem provides an invariant splitting in the absence of
invertibility of the operators as long as the base is invertible. It is
natural to ask whether one can obtain an invariant splitting in the
absence of invertibility of the base. The following example establishes
that in general this is not possible.

\begin{example}
Let $\Sigma=\{0,1\}^{\mathbb Z}$ be equipped with the
shift-transformation $\sigma$ and the $(\frac12,\frac12)$-Bernoulli
measure and let $A_0$ and $A_1$ be two non-commuting invertible $2\times
2$ matrices which we consider as operators on $\mathbb R^2$. Let
$\mathcal L_\omega\colon \mathbb R^2\to\mathbb R^2$ be given by $\mathcal
L_\omega=A_{\omega_0}$. We assume further that the two Lyapunov exponents
of the random dynamical system differ.  As is standard we define for
$n>0$, $\mathcal L^{(-n)}_\omega=A_{\omega_{-n}}^{-1}\circ \cdots\circ
A_{\omega_{-1}}^{-1}$. We call this random dynamical system $\mathcal R$.

Oseledets' theorem then guarantees that there is a decomposition
$\mathbb R^2=E_1(\omega)\oplus E_2(\omega)$ such that for $v\in
E_i(\omega)\setminus \{0\}$, $(1/n)\log\|\mathcal L^{(n)}_\omega
v\|\to \lambda_i$ both as $n\to\infty$ and as $n\to-\infty$; and
$\mathcal L_\omega(E_i(\omega))=E_i(\sigma(\omega))$ and $\mathcal
L_{\sigma^{-1}\omega}^{-1}(E_i(\omega))=E_i(\sigma^{-1}\omega)$ for
almost every $\omega$. By uniqueness (Theorem \ref{thm:main}), this splitting
is unique.

We define an inverse system $\overline{\mathcal R}$ as follows:
$\overline\Sigma=\Sigma$ where the base map is
$\overline\sigma=\sigma^{-1}$.  We define the operators on this
inverse system by $\overline{\mathcal
  L}_\omega=A_{\omega_{-1}}^{-1}$. Oseledets' theorem guarantees that
the splitting $E_1(\omega)\oplus E_2(\omega)$ also works for $\mathcal
R$.

We now define non-invertible systems $\mathcal R^+$ and $\mathcal R^-$
obtained by truncating the shifts in $\mathcal R$ and $\overline
{\mathcal R}$ to one-sided shifts. $\Sigma^+$ is defined to be
$\{0,1\}^{\mathbb Z^+}$ (where $\mathbb Z^+=\{0,1,2,\ldots\}$) and
$\Sigma^-$ is defined to be $\{0,1\}^{\mathbb Z^-}$ (where $\mathbb
Z^-=\{-1,-2,\ldots\}$). Let $\Sigma^+$ and $\Sigma^-$ be equipped with
their Borel $\sigma$-algebras $\mathcal B^+$ and $\mathcal B^-$. The
maps $\sigma$ and $\overline\sigma$ factor naturally onto maps
$\sigma^+$ on $\Sigma^+$ and $\sigma^-$ on $\Sigma^-$ through
$\pi^+(\omega)=(\omega_n)_{n\in\mathbb Z^+}$ and
$\pi^-(\omega)=(\omega_n)_{n\in\mathbb Z^-}$. Define $\mathcal L^+_\eta$
for $\eta\in\Sigma^+=A_{\eta_0}$ as before and similarly
$\mathcal L^-_\xi=A^{-1}_{\xi_{-1}}$ for $\xi\in\Sigma^-$ as
before. Let $\mathcal R^+$ and $\mathcal R^-$ be the two
one-sided dynamical systems.

Suppose now for a contradiction that there are Oseledets splittings
for $\mathcal R^+$ and $\mathcal R^-$: $E_1^+(\eta)\oplus
E_2^+(\eta)$ for $\eta\in\Sigma^+$ and $E_1^-(\xi)\oplus E_2^-(\xi)$
for $\xi\in\Sigma^-$ respectively.

Then one can check that $E_1^+(\pi^+(\omega))\oplus
E_2^+(\pi^+(\omega)) $ is an invariant splitting for $\mathcal R$
which gives the correct rates of expansion as $n\to\infty$.  Theorem
\ref{thm:main} guarantees that there is only one such splitting and hence we
see that
\begin{equation}\label{eq:plus}
E_i(\omega)=E_i^+(\pi^+(\omega))\text{ for almost every $\omega$.}
\end{equation}
Similarly $E_1^-(\pi^-(\omega))\oplus E_2^-(\pi^-(\omega))$ is an
invariant splitting for $\overline{\mathcal R}$ which gives the
correct rates of expansion as the power $n$ of the inverse random
dynamical system approaches $\infty$. Since the splitting for
$\overline{\mathcal R}$ was the same as that for $\mathcal R$ we
deduce that
\begin{equation}\label{eq:minus}
E_i(\omega)=E_i^-(\pi^-(\omega))\text{ for almost every $\omega$.}
\end{equation}
From \eqref{eq:plus} we deduce that $E_i$ is $\mathcal
F^+$-measurable where $\mathcal F^+={\pi^+}^{-1}\mathcal B^+$ whereas
from \eqref{eq:minus} we deduce that $E_i$ is $\mathcal
F^-$-measurable where $\mathcal F^-={\pi^-}^{-1}\mathcal B^-$. It
follows that the $E_i$ are $\mathcal F^-\cap \mathcal
F^+$-measurable. Since the intersection $\mathcal F^-\cap \mathcal
F^+$ is the trivial sigma-algebra it follows that $E_i$ is
constant almost everywhere, equal to $E^*_i$ say. From this it follows
that $A_0(E^*_i)=A_1(E^*_i)=E^*_i$ so that $A_0$ and $A_1$ have common
eigenspaces and hence are simultaneously diagonalizable. Since they do
not commute by assumption this is a contradiction.
\end{example}

\subsection{Reduction to the invertible case in Thieullen's
Theorem}\label{subsec:Thieullensuff}

As mentioned above Thieullen deduces the non-invertible version of his
theorem from the invertible case by constructing an invertible extension
of the given system. More specifically if the original system has maps
$\mathcal L_\omega$ acting on a Banach space $X$ the new system has maps
$\tilde {\mathcal L_\omega}$ acting on a Banach space $\tilde X$ where
\begin{align*}
  &\tilde X=\{(x_0,x_1,\ldots)\colon x_i\in X,\,
  \sup\|x_i\|<\infty\}\text{; and}\\
  &\tilde {\mathcal L_\omega}(x_0,x_1,\ldots)=(\mathcal
  L_\omega(x_0),\alpha_0x_0,\alpha_1x_1,\ldots).
\end{align*}

Thieullen then defines $\gamma_n=\sum_{k\le n}\log\alpha_k$ and states
conditions on the $(\alpha_n)$ and $(\gamma_n)$ which suffice to
ensure that the exceptional spectrum of the extension agrees with the
exceptional spectrum of the original system.

His conditions are as follows:
\begin{enumerate}
  \item $(\alpha_n)$ is a strictly decreasing sequence converging to
    0;\label{item:cond1}
  \item $\lim_{n\to\infty} \gamma_n/n=-\infty$;\label{item:cond2}
  \item $\forall\mu<0, \sup\{p\ge 0\colon \gamma_p\ge (n+p)\mu\}=o(n).$
    \label{item:cond3}
\end{enumerate}

We claim that Condition \eqref{item:cond1} implies the other two
conditions. That \eqref{item:cond1} implies \eqref{item:cond2} is
immediate. We now indicate a brief proof that \eqref{item:cond2}
implies \eqref{item:cond3}.

Let $\mu<0$ and $\epsilon>0$ be given. Set $M=-\mu/\epsilon-\mu$. By
\eqref{item:cond2} there is a $p_0$ such that for $p\ge p_0$ we have
$\gamma_p/p<-M$. At this point, choose an $n$.
If $\gamma_p\ge(n+p)\mu$ 
then either (i) $p<p_0$ or (ii) $p\ge p_0$ and therefore $(n+p)\mu\le\gamma_p<-pM$.
In the latter case $-p\mu/\epsilon=p(M+\mu)\le -n\mu$. We then see that
$p\le\min(p_0,\epsilon n)$. Since $\epsilon$ is arbitrary we see that
$\sup\{p\ge 0\colon \gamma_p\ge (n+p)\mu\}=o(n)$ as required and
Condition \eqref{item:cond3} is established.

We remark that in Thieullen's proofs it is sufficient to take a sequence
$(\alpha_k)$ for which \eqref{item:cond2} is satisfied. Clearly the most
natural way to do this is to take any sequence satisfying
\eqref{item:cond1}.

\section{Applications}\label{sec:apps}

The motivation for the development of Theorem \ref{thm:FLQ} is the
desire to extend transfer operator approaches for the global
analysis of dynamical systems from deterministic autonomous
dynamical systems to random or non-autonomous dynamical systems.

A common setting for deterministic systems is: $M\subset
\mathbb{R}^m$ is a smooth manifold and $T:M\to M$ a $C^1$
map with some additional regularity properties.  The (deterministic)
dynamical system $T:M\to M$ has an associated {\em
Perron--Frobenius operator} $\mathcal{L}_T:X\to X$ defined
by $\mathcal{L}_Tf(x)=\sum_{y\in T^{-1}x} f(y)/|\det DT(y)|$, where
$X$ is a Banach space of complex-valued functions on $M$. The
Perron--Frobenius operator evolves density functions on $M$ forward
in time, just as the map $T$ evolves single points $x\in M$ forward
in time.

More generally, the ``weight'' $1/|\det DT(y)|$ may be replaced with
a sufficiently regular generalised weight $g(y)$ to form a {\em
transfer operator}. Perron--Frobenius operators and transfer
operators have proven to be indispensable tools for studying the
long term behaviour of dynamical systems.  An ergodic absolutely
continuous invariant probability measure (ACIP) describes the long
term distribution of forward trajectories $\{T^kx\}_{k=0}^\infty$ in
$M$ for Lebesgue almost-all initial points in $x\in M$. An early use
of Perron--Frobenius operators was to prove the existence of ACIMs
for piecewise $C^2$ expanding maps \cite{LasotaYorke}. A study of
the peripheral spectrum of $\mathcal{L}_T$ yielded information on
the number of ergodic ACIPs \cite{HofbauerKeller,Rychlik}. The
particular weight function $1/|\det DT(y)|$ is attuned to ACIPs.
Other ``equilibrium states'' can be read off from the leading
eigenfunction of the transfer operator by varying the weight
function $g$ (in statistical mechanics terms, $g$ describes the
local energy of states in $M$).

The spectrum of the Perron--Frobenius operator provides information
on the exponential rate at which observables become temporally
decorrelated.  The essential spectral radius of Perron--Frobenius
operators \cite{Keller} establishes a threshold beyond which
spectral values are necessarily isolated.  Furthermore, this radius
is typically connected with the average rate at which nearby
trajectories separate.  Thus, these isolated spectral values are of
particular interest in applications because they predict
decorrelation rates slower than one expects to be produced by local
separation of trajectories.  The eigenfunctions associated with
these isolated eigenvalues have been used to detect slowly mixing
structures in a variety of physical systems, see, for example, \
\cite{Schuette,DellnitzMarsden,FPET07,DFHPS09}.

From a physical applications point of view, it is natural to study
random or time-dependent (non-autonomous) dynamical systems using a
transfer operator methodology. Theorem \ref{thm:FLQ} considered
this question in the setting of a finite number of piecewise linear,
expanding interval maps, sharing a joint Markov partition, where the
Perron--Frobenius operators acted on the space of functions of
bounded variation. In the present work, in our first application, we
remove the assumptions of finiteness,
piecewise linearity and Markovness, and
allow random compositions that are expanding-on-average. Our
second application is to subshifts of finite type with random continuously-parametrized weight functions.

\subsection{Application I: Interval maps.}

We now show that Theorem \ref{thm:main} can be applied in the context of
random compositions of expanding-on-average mappings acting through their
Perron--Frobenius operators on the space $\mathrm{BV}$ of functions of
bounded variation. In this context a major drawback of the Thieullen
approach becomes clear: if $T_1$ and $T_2$ are \emph{any} two distinct
expanding mappings then their Perron--Frobenius operators $\mathcal
L_{T_1}$ and $\mathcal L_{T_2}$ are far apart in the operator norm on
$\mathrm{BV}$. In fact the set of Perron--Frobenius operators acting on
$\mathrm{BV}$ is discrete. As a consequence, in order for $\omega\mapsto
\mathcal L_{T_\omega}$ to be a continuous map on a compact space, the
maps range of $\omega\mapsto T_\omega$ is forced to be finite. If we want
$\omega\mapsto \mathcal L_{T_\omega}$ to be $\mathbb{P}$-continuous then
it can have at most countable range.

Let $I=[0,1]\subset \mathbb{R}$ denote the closed unit interval,
$\mathcal B$ denote the Borel $\sigma$-algebra and $m$ denote Lebesgue
measure.

\begin{definition}
We say a map $T:I\to I$ is \emph{Rychlik} if
\begin{enumerate}
\item $T$ is differentiable on a dense open subset $U_T\subset I$ of full
measure;
\item for each connected component $B$ of $U_T$, $T|_B$ extends
to a homeomorphism from $\overline{B}$ to a subinterval of $I$;
\item the function $g_T:I\to \mathbb{R}$ has bounded variation, where
\begin{equation*}
g_T(x) = \left\{
\begin{array}{cl}
\frac{1}{|DT(x)|} & x\in U_T \\
0 & \mathrm{otherwise.}
\end{array}\right.
\end{equation*}
\end{enumerate}
\end{definition}

The class of Rychlik maps is closed under composition. Recall that the
\emph{variation} of a function $f:I\to \mathbb{R}$ is the quantity
\begin{equation}
\var(f) := \sup_{0=p_1<p_2<\ldots<p_k=1} \sum_{i=1}^k | f(p_i)-f(p_{i-1})
|.
\end{equation}
A function on the interval is said to be of bounded variation if
$\var(f)<\infty$.

The \emph{Perron--Frobenius operator} for a Rychlik map $T$ is defined,
for a function $f\in L^1(I)$ by
\begin{equation}
\mathcal{L}_Tf(x) = \sum_{y\in T^{-1}(x)} g_T(y) f(y).
\end{equation}
The Perron--Frobenius operator is a \emph{Markov operator}: that is, if
$f\in L^1(I)$, then $\int \mathcal L_T f\,\mathrm dm = \int f\,\mathrm
dm$, and if $f\geq 0$, then $\mathcal L_T f\geq 0$.

We consider the action of $\mathcal{L}_T$ on the Banach space
\begin{equation*}
\text{BV}:= \left\{ f\in L^\infty(I) : f\text{ has a version }
\tilde{f}\text{ with }\var \tilde{f}<\infty   \right\}
\end{equation*}
with norm $\|f\|:=\max(\|f\|_1, \inf\{\var(\tilde{f}):\tilde{f}\text{ is
a version of }f \})$. A version $\tilde{f}$ of $f\in\mathrm{BV}$ has
minimal variation if and only if $\tilde{f}(x)\in [\lim_{y\to
x^-}f(y),\lim_{y\to x^+}f(y)]$ for all $x$. We shall assume versions are
chosen so as to satisfy this condition, unless stated otherwise.

We shall need a lemma that is a combination of Lemmas 4, 5 and 6 from
Rychlik \cite{Rychlik}.

\begin{lemma}[Rychlik \cite{Rychlik}]\label{lem:rychlik}
Let $T$ be a Rychlik map of the unit interval and let $\mathcal L_T$ be
its Perron-Frobenius operator. Suppose $\essinf_x |T'(x)|>1$. Let
$a=3/\essinf |T'|$. Then there is a partition $\mathcal P$ of the
unit interval into finitely many subintervals and a constant $D$ such
that for all $f\in BV$
\begin{equation*}
  \var \mathcal{L}_T f \leq a \var f +D\sum_{J\in\mathcal P}
  \left|\int_J f\right|.
\end{equation*}
\end{lemma}

We define a random composition of Rychlik maps as follows. Let
$\{T_i\}_{i\in I}$, be a finite or countably infinite set of Rychlik maps.  Let
$\overline I$ denote the one-point compactification of $I$ (with the
discrete topology) and let $S=\overline I^{\mathbb Z}$. Let
$\sigma\colon S\to S$ be the shift map and let $\mathbb P$ be an
ergodic shift-invariant probability measure supported on
$\Omega=I^{\mathbb Z}$.  For $\omega\in\Omega$ let $\mathcal
L_\omega=\mathcal L_{T_{\omega_0}}$ be the Perron--Frobenius operator
of the map $T_{\omega_0}$ acting on the space $\text{BV}$. We make the
further assumption that $\int\log^+\|\mathcal L_\omega\|\,d\mathbb
P(\omega)<\infty$ (or equivalently $\sum_{i\in I}\mathbb
P(\{i\})\log\|\mathcal L_{T_i}\|<\infty$). If these conditions are
satisfied we refer to the 6-tuple $\mathcal R=(\Omega,\mathcal
F,\mathbb P,\sigma,\text{BV},\mathcal L)$ as a \emph{Rychlik random
  dynamical system}.

One can then verify that the system $\mathcal R$ satisfies the
assumptions of Theorem \ref{thm:main}.

We denote the $n$-fold composition
$T_{\sigma^{n-1}\omega}\circ \cdots T_{\sigma\omega}\circ T_\omega$ by
$T_\omega^{(n)}$. It is well known that
the composition, $\mathcal{L}^{(n)}_\omega$, of the Perron--Frobenius operators of
$T_{\omega}, T_{\sigma\omega},\ldots,T_{\sigma^{n-1}\omega}$ is
equal to the Perron--Frobenius operator of $T^{(n)}_\omega$. A random
composition may also be considered as a single transformation on the
space $\Omega\times I$ which we endow with the sigma-algebra $\mathcal
F\otimes \mathcal B$: the \emph{skew product} $\Theta:\Omega\times
I\to\Omega\times I$ is given by $ \Theta(\omega,x)=(\sigma\omega,
T_\omega x)$.

We shall need a well-known inequality relating the index of compactness
to the essential spectral radius. For a version of the converse
inequality the reader is referred to work of Morris \cite{Morris}. Let
$A\colon X\to X$ be a linear operator on a Banach space. We write
$\|A\|_\text{fr}$ for $\inf\{\|A-F\|\colon\text{$F$ has finite rank}\}$.
Recall from earlier $\|A\|_\text{ic}$ is defined to be $\inf\{r\colon
A(B_X)$ may be covered by a finite number of $r$-balls$\}$.

\begin{lemma}\label{lem:morris}
For a linear operator $A$ between Banach spaces $\|A\|_\text{ic}\le
\|A\|_\text{fr}$.
\end{lemma}
\begin{proof}
Let $A=F+R$ where $F$ has finite rank and $\|R\|=r$. Let $\epsilon>0$.
Since $F(B_X)$ is compact it may be covered by a finite number of
$\epsilon$-balls for any $\epsilon>0$, $\bigcup_{n=1}^N B_\epsilon(x_n)$.
Hence $A(B_X)\subset F(B_X)+R(B_X)\subset \bigcup_{n=1}^N
B_\epsilon(x_n)+B_r(0)=\bigcup_{n=1}^N B_{r+\epsilon}(x_n)$ so that for
each $\epsilon>0$, $\|A\|_\text{ic}\le r+\epsilon$. Since it is possible
to find decompositions with $r$ arbitrarily close to $\|A\|_\text{fr}$
the lemma follows.
\end{proof}

Keller \cite{Keller} used Lemma \ref{lem:rychlik} together with a
supplementary argument to identify the essential spectral radius of the
Perron--Frobenius operator of an expanding Rychlik map acting on the space
of functions of bounded variation. We show that Keller's argument applies
equally in our context of random dynamical systems.

\begin{theorem}
  Let $\mathcal R=(\Omega,\mathcal F,\mathbb
  P,\sigma,\text{BV},\mathcal L)$ be a Rychlik random dynamical
  system. Then there exists a $\chi$ such that for $\mathbb{P}$-almost
  every $\omega$,
  \begin{align}
  \left(1/\essinf\left|{T^{(n)}_\omega}'(x)\right|\right)^{1/n}
  \rightarrow\chi.
  \end{align}
  Further if $\chi<1$ then $\|\mathcal L^{(n)}_\omega\|_{\text{ic}}
  ^{1/n}\rightarrow \chi$.
\end{theorem}

\begin{definition}
  We say that the Rychlik random dynamical system appearing in the theorem is
  \emph{expanding-on-average} if $\chi<1$.
\end{definition}

\begin{proof}
We note that both $\|\mathcal L^{(n)}_\omega\|_\text{ic}$ and
$a_n(\omega)=1/\essinf_x\left|{T^{(n)}_\omega}'(x)\right|$ are
submultiplicative. It follows from the subadditive ergodic theorem that
both of the limits appearing in the statement of the theorem exist for
$\mathbb P$-almost every $\omega$. In the case where $\chi<1$ we claim
the following inequalities:

\begin{equation}\label{eq:essrad}
  a_n(\omega)\le \left\|\mathcal L^{(n)}_\omega\right\|_{\text{ic}}
  \le \left\|\mathcal L^{(n)}_\omega\right\|_{\text{fr}}\le
  3a_n(\omega)\text{ provided $a_n(\omega)<1$.}
\end{equation}

The middle inequality is Lemma \ref{lem:morris}. To see the upper
bound, let $\mathcal P$ be the partition of the interval into
subintervals guaranteed by Lemma \ref{lem:rychlik}. Let $E_\mathcal P$
be the conditional expectation operator defined by
\begin{equation*}
  E_\mathcal P f(t)=\frac{1}{|J|}\int_J f\text{ for $t\in J$}.
\end{equation*}
We then have $\mathcal L^{(n)}_\omega=\mathcal L^{(n)}_\omega\circ
(1-E_\mathcal P)+\mathcal L^{(n)}_\omega\circ E_\mathcal P$. The second
term has finite rank and Lemma \ref{lem:rychlik} guarantees that $\var
(\mathcal L^{(n)}_\omega\circ (1-E_\mathcal P)f)\le 3a_n(\omega) \var f$.
Since $\mathcal L^{(n)}_\omega$ preserves integrals and $(1-E_\mathcal
P)f$ has integral 0, it follows that $\mathcal L^{(n)}_\omega\circ
(1-E_\mathcal P)f$ has integral 0 and therefore that the $L^1$ norm is
bounded above by the variation. This yields $\|\mathcal
L^{(n)}_\omega\circ (1-E_\mathcal P)\|\le 3a_n(\omega)$ so that
$\|\mathcal L^{(n)}_\omega\|_{\text{fr}}\le 3a_n(\omega)$.

For the lower bound fix an $\epsilon>0$ and suppose that
$1/|{T^{(n)}_\omega}'(x)|>(1-\epsilon)a_n(\omega)$ for $x$ in an
interval $J$. Suppose further that $J$ lies in a single branch of
$T^{(n)}_\omega$. Let $I$ and $I'$ be two subintervals of $J$, with no
endpoints in common and let $f_I=\frac12\mathbf 1_{I}$ and
$f_{I'}=\frac12\mathbf 1_{I'}$. Then we have $\|\mathcal
L^{(n)}_\omega f_I-\mathcal L^{(n)}_\omega f_{I'}\|
>2(1-\epsilon)a_n(\omega)$. It follows that no
$(1-\epsilon)a_n(\omega)$ ball contains more than two $\mathcal
L^{(n)}_\omega f_I$'s with distinct endpoints and so in particular
$\mathcal L^{(n)}_\omega B_{\mathrm{BV}}$ does not have a finite cover
by $(1-\epsilon)a_n(\omega)$ balls. We see that $\left\|\mathcal
  L^{(n)}_\omega\right\|_{\text{ic}}\ge
(1-\epsilon)a_n(\omega)$. Since $\epsilon$ is arbitrary, we see that
\eqref{eq:essrad} follows.

Taking $n$th roots and taking the limit, the theorem follows.
\end{proof}

We now demonstrate that $\lambda^*=0$. As a Perron--Frobenius operator
$\mathcal L^{(n)}_\omega$ is a stochastic operator for each
$\omega\in\Omega$, for any density $0\not\equiv f\in\mathrm{BV}$ we have
$\|\mathcal L^{(n)}_\omega f\|\geq \|\mathcal L^{(n)}_\omega
f\|_1=\|f\|_1$, which shows that $\lambda(\omega)\geq 0$. To show
$\lambda^*\leq 0$, since $\|\mathcal L_\omega\|_1\leq 1$, it suffices to
consider the growth of the variation of $\mathcal L^{(n)}_\omega f$.  As
$\chi<1$, for almost every $\omega\in\Omega$ there exists $n\in\mathbb
N$, $0<\alpha<1$ and $\beta\geq 0$ such that $\var \mathcal
L^{(n)}_\omega f \leq \alpha \var f +\beta\|f\|_1$ by Lemma
\ref{lem:rychlik}. Iterating this inequality gives a bound for the
sequence $(\var \mathcal L^{(kn)}_\omega)_{k\in\mathbb N}$, and so
$\liminf_{k\to\infty} (1/(nk))\log\|\mathcal L^{(nk)}_\omega\|\leq 0$.
As $\lim_{n\to\infty} (1/n)\log\|\mathcal L^{(n)}_\omega\|$ exists for
$\mathbb P$-almost every $\omega$, we have $\lambda^*\leq 0$.

\begin{corollary}\label{cor:RychlikOS}
  Let $\mathcal R=(\Omega,\mathcal F,\mathbb
  P,\sigma,\text{BV},\mathcal L)$ be a Rychlik random dynamical
  system. Assume that $\mathcal R$ is expanding-on-average. Then
  $\mathcal R$ is quasi-compact, with
\begin{equation*}
  \kappa^* = \lim_{n\to\infty}\frac 1n\log \left(1/\essinf_x
    \left|{T^{(n)}_\omega}'(x)\right|\right)<0=\lambda^*\text{ for
    $\mathbb{P}$-almost every $\omega$}.
\end{equation*}
The random dynamical system therefore admits a $\mathbb P$-continuous
Oseledets splitting.
\end{corollary}

The Oseledets splitting provides information on the invariant measures
and rates of mixing of the random system. A natural generalisation of
the notion of `invariant measure' to the random setting is the concept
of `sample measure'. A family $\{\mu_\omega\}_{\omega\in\Omega}$ of
\emph{sample measures} (see \cite{Arnold}), is a family of probability
measures $\mu_\omega$ on $I$ satisfying
\begin{enumerate}
\item for all $U\in\mathcal F$, the map $\omega\mapsto \mu_\omega(U)$
  is $\mathcal F$-measurable.
\item $T_\omega \mu_\omega = \mu_{\sigma\omega}$ for
  a.\,e. $\omega\in\Omega$.
\end{enumerate}
Given a family $\{\mu_\omega\}_{\omega\in\Omega}$ of sample measures,
the measure $\mu$ on $\Omega\times I$ given by $\mathrm d\mu(\omega,x)
:= \mathrm d\mu_\omega(x) \mathrm d\mathbb P(\omega)$ is an invariant
probability for the associated skew product
$\Theta(\omega,x)=(\sigma\omega,T_\omega x)$. Conversely, any
$\Theta$-invariant probability measure $\mu$ with marginal $\mathbb P$
on $\Omega$ may be disintegrated to give a family of sample measures
for the original system.

Sample measures for random compositions of expanding interval maps have
previously been studied by Pelikan \cite{Pelikan}, Morita \cite{Morita},
and in a more general setting by Buzzi \cite{Buzzi}. He considers random
compositions of Lasota--Yorke maps that have neither too many branches
nor too large distortion, and proves that the associated skew product
transformation possesses a finite number of mutually singular ergodic
ACIPs $\mu$, each giving a family $\{\mu_\omega\}_{\omega\in\Omega}$ of
sample measures with densities of bounded variation. Returning to the
present setting of a random composition of Rychlik maps, any such family $\{f_\omega\}_{\omega\in\Omega}$ of sample
measures with densities of bounded variation satisfies $\mathrm
d\mu_\omega/\mathrm dm\in E_1(\omega)$ for $\mathbb P$-almost every
$\omega$. It follows that the number of such mutually singular ergodic
ACIPs (whose sample measure densities are necessarily linearly
independent for $\mathbb P$-a.\,e. $\omega$) is bounded by $d_1$, the
dimension of the Oseledets subspace $E_1(\omega)$.

Furthermore, the exceptional Lyapunov spectral values strictly less than
$0$, and their corresponding Oseledets subspaces, provide information on
exponential decay rates that are slower than the decay produced by local
separation of trajectories. The authors discuss and provide examples of
such spectral values and Oseledets subspaces in \cite{FLQ}. Corollary
\ref{cor:RychlikOS} provides conditions under which Oseledets subspaces
exist in much greater generality than in \cite{FLQ}, removing the
assumptions of piecewise linearity and Markovness, and allowing the
system to be expanding on average. In non-rigorous numerical experiments,
Oseledets subspaces have been shown to effectively capture so-called
``coherent sets'' in aperiodic fluid flow \cite{FLS}. The present work
represents a first step toward making such calculations rigorous by
extending the study of Perron--Frobenius operator cocycles to Banach
spaces that are more representative of fluid flow.

\subsection{Application II: Transfer Operators with Random Weights}

Let $\Sigma$ be a one-sided 1-step shift of finite type on $N$ symbols.
We assume that for each symbol $j$ in the alphabet there is at least one
$i$ for which $ij$ is a legal transition (if not we restrict our
attention to the subset of $\Sigma$ obtained by deleting all symbols that
have no preimage).  For $x,y\in\Sigma$ we let $\Delta(x,y)$ be
$\min\{n\colon x_n\ne y_n\}$ (or $\infty$ if $x=y$). The $\theta$-metric
on $\Sigma$ is $d_\theta(x,y)=\theta^{\Delta(x,y)}$ (so that the standard
metric is $d_{1/2}$).

We will write $S$ for the usual left shift map on $\Sigma$. If
$x\in\Sigma$ and $v$ is a word of some length $k\ge 1$ in the alphabet
such that $v_{k-1}x_0$ is a legal transition then we will write $vx$ for
the point in $S^{-k}x$ obtained by concatenating $v$ and $x$.

Let $\mathcal C_\theta$ denote the set of $\theta$-Lipschitz functions:
those functions $f$ for which there is a $C$ such that $|f(x)-f(y)|\le
Cd_\theta(x,y)$ for all $x$ and $y$. We define $|f|_\theta$ to be the
smallest $C$ for which such an inequality holds. As usual we endow
$\mathcal{C}_\theta$ with the topology generated by the norm
$\|f\|_\theta=\max(|f|_\theta,\|f\|_\infty)$. Let $\mathcal{W}_\theta$ be
the collection of those functions $g$ in $\mathcal C_\theta$ such that
$\min_x g(x)>0$.

Denote $P_gf(x)=\sum_{y\in S^{-1}x}f(y)g(y)$ and consider $P_g$ as
an operator on $(\mathcal C_\theta,\|\cdot\|_\theta)$. For the purposes
of the following lemma we consider arbitrary $g\in \mathcal C_\theta$ but
we shall later restrict to $g\in \mathcal W_\theta$.

\begin{lemma}
The map $P\colon \mathcal  C_\theta\to L(\mathcal C_\theta,\mathcal
C_\theta)$ is continuous with respect to the operator norm on $L(\mathcal
C_\theta,\mathcal C_\theta)$.
\end{lemma}
\begin{proof}
$P$ is clearly linear. Let $g\in\mathcal{C}$. We want to bound $\|P_g
f\|_\theta$. We first estimate $\|P_gf\|_\infty$. Let $x\in\Sigma$. Then
$|P_gf(x)|\le \sum_{y\in\sigma^{-1}x}|f(y)|\cdot |g(y)|\le
N\|f\|_\infty\|g\|_\infty\le N\|f\|_\theta\|g\|_\theta$. This yields
\begin{equation}\label{eq:term1}
\|P_gf\|_\infty\le N\|f\|_\theta\|g\|_\theta.
\end{equation}

We now bound $|P_gf|_\theta$. Let $x\ne y\in\Sigma$. We need to estimate
$|P_gf(x)-P_gf(y)|/d_\theta(x,y)$. If $x_0\ne y_0$ then the denominator
is 1 and the numerator is at most $2N\|g\|_\infty\|f\|_\infty\le
2N\|g\|_\theta\|f\|_\theta$. If $x_0=y_0$ then
\begin{align*}
|P_gf(x)-P_gf(y)|&=\sum_{\{i\colon ix_0\text{
legal}\}}\big(g(ix)f(ix)-g(iy)f(iy)\big)\\
&\le N(\|g\|_\infty|f|_\theta+\|f\|_\infty|g|_\theta)
d_\theta(ix,iy)\\
&\le 2 N\|g\|_\theta\|f\|_\theta \theta d_\theta(x,y).
\end{align*}
Combined with the estimate in the case $x_0\ne y_0$, this shows
$|P_gf|_\theta\le 2N\|g\|_\theta\cdot\|f\|_\theta$ and so $\|P_g\|\le
2N\|g\|_\theta$.
\end{proof}

Baladi's book \cite{Baladi} contains a number of detailed calculations of
the spectral radii and essential spectral radii of Perron--Frobenius
operators acting on the Lipschitz spaces. We now develop some of these
arguments in the case of random compositions.

Suppose that $G\colon \Omega\mapsto \mathcal W_\theta;\;\omega\mapsto
g_\omega$ is a continuous mapping. Since $\Omega$ will be assumed to be
compact there will be a constant $\gamma$ such that $g_\omega(x)\ge
\gamma$ for all $x\in\Sigma$ and $\omega\in\Omega$. Similarly there will
be a constant such that $\|g_\omega\|_\theta\le C$ for all
$\omega\in\Omega$. We assume as usual that $\sigma\colon\Omega\to\Omega$
is ergodic. We write $P^{(n)}_\omega$ for the composition of
Perron-Frobenius operators $P_{g_{\sigma^{n-1}\omega}}\circ\cdots\circ
P_{g_\omega}$.

A linear map on $\mathcal C_\theta$ is said to be \emph{positive} if it
maps non-negative functions to non-negative functions. In particular if
$g\in \mathcal W_\theta$ then $P_g$ is positive.

\begin{lemma}
Let $\mathcal R=(\Omega,\mathcal F,\mathbb{P},\sigma,\mathcal
C_\theta,P)$ be a continuous ergodic random dynamical system of
Perron-Frobenius operators with random weights on a shift of finite type
$\Sigma$. Suppose that $\Omega$ is compact and $P\colon\Omega\to \mathcal
W_\theta$ is continuous. Let $R_n(\omega)=\|P^{(n)}_\omega\mathbf
1\|_\infty$. Then $R_n(\omega)^{1/n}$ converges $\mathbb{P}$-almost
everywhere to a constant $R^*$.
\end{lemma}

\begin{proof}
Since the operators $P_{g}$ are positive ($g$ being positive), we have
$P^{(n+m)}_\omega \mathbf
1=P^{(n)}_{\sigma^m\omega}(P^{(m)}_\omega\mathbf 1)\le
P^{(n)}_{\sigma^m\omega} R_m(\omega)\mathbf 1\le
R_n(\sigma^m\omega)R_m(\omega)\mathbf 1$. It follows that $\log
R_n(\omega)$ is a subadditive sequence of functions so that by the
subadditive ergodic theorem, for $\mathbb{P}$-almost all $\omega$,
$R_n(\omega)^{1/n}$ converges to a quantity $R(\omega)$. Since this
quantity is $\sigma$-invariant, there is a constant $R^*$ such that
$R(\omega)=R^*$ for $\mathbb{P}$-almost every $\omega\in\Omega$.
\end{proof}

\begin{lemma}[Bounded Distortion]
Let $\mathcal R$ be as in the previous lemma. Let $g^{(n)}_\omega(x)$
denote the product $g_\omega(x)g_{\sigma(\omega)}(Sx)\ldots
g_{\sigma^{n-1}\omega}(S^{n-1}x)$. There exists a $D>0$ such that for all
$\omega\in\Omega$, if $x_0=y_0$ and $v$ is a word of an arbitrary length
$k$ such that $v_{k-1}x_0$ is a legal transition then
\begin{equation*}
  \left|1-\frac{g^{(k)}_\omega(vy)}{g^{(k)}_\omega(vx)}\right|
  \le Dd_\theta(x,y).
\end{equation*}
\end{lemma}

\begin{proof}
As mentioned above there is a $\gamma>0$ such that $g_\omega(x)\ge
\gamma$ for all $\omega\in\Omega$ and all $x\in\Sigma$. Similarly there
is a $\Gamma$ such that $g_\omega(x)\le\Gamma$ for all $\omega$ and $x$
and also a $C$ such that $\|g_\omega\|_\theta\le C$ for all
$\omega\in\Omega$. We make use of the fact that there exists a constant
$K$ such that if $\gamma<a,b<\Gamma$ then $|\log(b/a)|\le K|b-a|$.

We then have
\begin{align*}
  \left|\log\frac{g^{(k)}_\omega(vy)}{g^{(k)}_\omega(vx)}\right|&\le
  \sum_{j=0}^{k-1}\left|\log
  \frac{g_{\sigma^j\omega}(S^j(vy))}{g_{\sigma^j\omega}(S^j(vx))}\right|\\
  &\le K\sum_{j=0}^{k-1}Cd_\theta(S^j(vx),S^j(vy))\\
  &=CK\sum_{j=0}^{k-1}\theta^{k-j}d_\theta(x,y)\\
  &\le (CK/(1-\theta))d_\theta(x,y),
\end{align*}

Exponentiating we see $g^{(k)}_\omega(vy)/g^{(k)}_\omega(vx)$ lies
between  the values $\exp(-rd_\theta(x,y))$ and $\exp(rd_\theta(x,y))$ where
$r=CK/(1-\theta)$. Since $\exp$ is Lipschitz on $[-r,r]$ there exists a
$D$ such that $|\exp(t)-1|\le (D/r)|t|$ on $[-r,r]$. It follows that
\begin{equation*}
\left|\frac{g^{(k)}_\omega(vy)}{g^{(k)}_\omega(vx)}-1\right| \le
Dd_\theta(x,y)
\end{equation*}
as required.
\end{proof}

The next lemma appears as an exercise in the deterministic case in
Baladi's book \cite{Baladi}.
\begin{lemma}\label{lem:LY}
  Let $\mathcal R$ be as above. Then there exists a constant $K$
  such that for $f\in \mathcal C_\theta$
  \begin{equation*}
    |P^{(n)}_\omega f|_\theta\le R_n(\omega)(\theta^n|f|_\theta+
    K\|f\|_\infty).
  \end{equation*}
\end{lemma}
\begin{proof}
  We need to estimate $\sup_{x\ne y}|P^{(n)}_\omega f(x)-P^{(n)}_\omega
  f(y)|/d_\theta(x,y)$. If $x$ and $y$ differ in the zeroth
  coordinate, the denominator is 1 and we bound the numerator above by
  $R_n(\omega)\|f\|_\infty$ giving a bound of the given form (with
  $K$=1).

If $x$ and $y$ agree in the zeroth coordinate then we estimate as
follows. We let $W_n$ be the set of words $v$ of length $n$ such that
$v_{n-1}x_0$ is legal.

  \begin{align*}
    &|P^{(n)}_\omega f(x)-P^{(n)}_\omega f(y)|\\
    &=
    \left|\sum_{v\in W_n}\big(g^{(n)}_\omega(vx)f(vx)-
    g^{(n)}_\omega(vy)f(vy)\big)\right|\\
    &\le \sum_{v\in W_n}g^{(n)}_\omega(vx)\cdot|f(vx)-f(vy)|+
    \sum_{v\in W_n}|f(vy)|\cdot|g^{(n)}_\omega(vx)-g^{(n)}_\omega(vy)|\\
    &\le \sum_{v\in W_n}g^{(n)}_\omega(vx)\left(|f|_\theta d_\theta(vx,vy)+
    \|f\|_\infty\left|1-\frac{g^{(n)}_\omega(vy)}{g^{(n)}_\omega(vx)}\right|\right)\\
    &\le R_n(\omega)\left(|f|_\theta\theta^nd_\theta(x,y)+\|f\|_\infty D
    d_\theta(x,y)
    \right).
  \end{align*}
  We therefore see that $|P^{(n)}_\omega
  f|_\theta\le
  R_n(\omega)\left(\theta^n|f|_\theta+D\|f\|_\infty\right)$ as required.
\end{proof}

Let $n>0$ and let $[w_1],\ldots,[w_k]$ be an enumeration of the
$n$-cylinders. For each $1\le j\le k$, let $x_j$ be a point of $[w_j]$.
Given these choices, define a finite rank operator $\Pi_n\colon \mathcal
C_\theta\to\mathcal C_\theta$ by
\begin{equation*}
  (\Pi_n f)(x)=f(x_j)\text{ for $x\in [w_j]$}.
\end{equation*}

\begin{lemma}\label{lem:proj}
  For $f\in\mathcal C_\theta$ and $\Pi_n$ as above we have
  \begin{align*}
    &\|(I-\Pi_n) f\|_\infty\le \theta^n |f|_\theta\\
    &|(I-\Pi_n) f|_\theta\le \max(2\theta,1)|f|_\theta.
  \end{align*}
\end{lemma}
\begin{proof}
Let $Q\colon \mathcal C_\theta\to\mathcal C_\theta$ denote $I-\Pi_n$. Let
$x\in [w_j]$. Then $Qf(x)=f(x)-f(x_j)$. Since $\Delta(x,x_j)\ge n$, we
have $|Qf(x)|\le |f|_\theta\theta^n$.

Now let $x,y\in\Sigma$. If they lie in the same $n$-cylinder set then
$|Qf(x)-Qf(y)|=|f(x)-f(y)|\le |f|_\theta d_\theta(x,y)$. On the other
hand if $x$ and $y$ lie in distinct $n$-cylinders, $[w_i]$ and $[w_j]$
respectively then we have
\begin{align*}
  &|Qf(x)-Qf(y)|=|(f(x)-f(x_i))-(f(y)-f(x_j))|\\
  &\le |f(x)-f(x_i)|+|f(y)-f(x_j)|.
\end{align*}
Since $\Delta(x,x_i)$ and $\Delta(y,x_j)$ are each at least $n$ the right
side is bounded above by $2|f|_\theta\theta^n\le (2\theta)|f|_\theta
d_\theta(x,y)$.
\end{proof}

\begin{theorem}\label{thm:tospec}
  Let $\mathcal R$ be as above. Then there are $C_1>0$ and $C_2>0$ such that
  \begin{align*}
    R_n(\omega)&\le \|P^{(n)}_\omega\|\le C_1R_n(\omega)\\
    (1/4)\theta^n R_n(\omega)&\le \|P^{(n)}_\omega\|_\text{ic}\le
    C_2\theta^n R_n(\omega).
  \end{align*}
In particular $\lambda(\omega)=R^*$ and $\kappa(\omega)=\theta R^*$ for
$\mathbb{P}$-almost every $\omega$ so that the random dynamical system is
quasi-compact.
\end{theorem}
\begin{proof}
  Let $K$ be as in Lemma \ref{lem:LY}.
  Let $f\in \mathcal C_\theta$. We have $|P^{(n)}_\omega f|_\theta \le
  (K+1)R_n(\omega)\|f\|_\theta$. Also $P^{(n)}_\omega f\le P^{(n)}_\omega
  (\|f\|_\infty\mathbf 1)\le R_n(\omega)\|f\|_\infty$. Combining these we
  see that $\|P^{(n)}_\omega f\|_\theta\le (K+1)R_n(\omega)\|f\|_\theta$.

  On the other hand we have $\|P^{(n)}_\omega\mathbf 1\|_\theta\ge
  R_n(\omega)$ while $\|\mathbf 1\|_\theta=1$ so the bounds on
  $\|P^{(n)}_\omega\|$ are established.

For the upper bound on $\|P^{(n)}_\omega\|_\text{ic}$ we use Lemma
\ref{lem:morris} to compare with $\|P^{(n)}_\omega\|_\text{fr}$ and we
let $\Pi_n$ be as above and give bounds on $\|P^{(n)}_\omega\circ
(I-\Pi_n)\|$. Let $f\in \mathcal C_\theta$. We have
$\|P^{(n)}_\omega\circ (I-\Pi_n)f\|_\infty\le
R_n(\omega)\|(I-\Pi_n)f\|_\infty \le \theta^nR_n(\omega)|f|_\theta$ where
the last inequality made use of Lemma \ref{lem:proj}. Using Lemmas
\ref{lem:LY} and \ref{lem:proj} we see
$|P^{(n)}_\omega((I-\Pi_n)f)|_\theta\le
R_n(\omega)(\theta^n\max(2\theta,1)|f|_\theta+K\theta^n\|f\|_\infty)$.
Combining these two inequalities leads to an upper bound of the desired
form for $\|P^{(n)}_\omega\|_\text{ic}$.

For the lower bound on $\|P^{(n)}_\omega\|_\text{ic}$ there exists (by
continuity) an open set $U$ on which $P^{(n)}_\omega\mathbf
1(x)>R_n(\omega)/2$. We show that the index of compactness is large by
exhibiting an infinite collection of points in the unit sphere of
$\mathcal C_\theta$ whose images under $P^{(n)}_\omega$ are uniformly
separated.

Let $u\in U$ and let $C_k$ be the $k$-cylinder about $u$. Since $U$ is
open there exists a $k_0$ such that $C_{k_0}\subset U$. Since $\Sigma$ is
an irreducible shift of finite type there exists an infinite sequence
$k_0<k_1<k_2<\ldots$ such that $C_{k_i}$ is a proper subset of
$C_{k_i-1}$ for all $i\ge 1$. We let $f_i=\theta^{k_i+n-1}\mathbf
1_{C_{k_i}}\circ S^n$. To check that $\|f_i\|_\theta=1$ we note that if
$x$ and $y$ agree on at least the first $k_i+n$ symbols then
$f_i(x)=f_i(y)$. Since the numerator of
$|f_i(x)-f_i(y)|/\theta^{\Delta(x,y)}$ takes only the values $0$ and
$\theta^{k_i+n-1}$ the maximum in this expression is obtained by taking
$x$ and $y$ that agree for as many symbols as possible, but for which
$f_i(x)\ne f_i(y)$. By the assumption on $\Sigma$ and choice of $k_i$
there are points agreeing for $k_i+n-1$ symbols but disagreeing on the
$k_i+n-1$st symbol for which $f_i(x)\ne f_i(y)$ so that
$\|f_i\|_\theta=1$ as required.

We then calculate
\begin{align*}
P^{(n)}_\omega f_i(x)&=\sum_{\{w_0^{n-1}\colon w_{n-1}x_0\text{ is
legal}\}}g^{(n)}_\omega(wx)\theta^{k_i+n-1}\mathbf 1_{C_{k_i}}\circ
S^n(wx)\\
&=\theta^{k_i+n-1}\mathbf 1_{C_{k_i}}(x)P^{(n)}_\omega \mathbf 1(x).
\end{align*}

Letting $h=P^{(n)}_\omega f_i-P^{(n)}_\omega f_j$, we have
$h=(\theta^{k_i+n-1}\mathbf 1_{C_{k_i}}-\theta^{k_j+n-1}\mathbf
1_{C_{k_j}})P^{(n)}_\omega\mathbf 1$. Let $i<j$ and pick $x\in
C_{k_j-1}\backslash C_{k_j}$ and $y\in C_{k_i-1}\backslash
C_{k_i}$. Then we have $\Delta(x,y)=k_i-1$,
$h(x)=\theta^{k_i+n-1}P^{(n)}_\omega \mathbf 1(x)\ge
(1/2)\theta^{k_i+n-1}R_n(\omega)$ whereas $h(y)=0$ giving
$\|h\|_\theta\ge (1/2)\theta^n R_n(\omega)$. It follows that no ball
of radius less than $(1/4)\theta^n R_n(\omega)$ can contain two
$P^{(n)}_\omega f_i$'s and so $\|P^{(n)}_\omega\|_\text{ic}\ge
(1/4)\theta^n R_n(\omega)$.

\end{proof}

\begin{example}
Let $\sigma\colon \Omega\to\Omega$ be any homeomorphic dynamical system
defined on a compact space $\Omega$ preserving an ergodic probability
measure $\mathbb P$. Let $\Sigma=\{0,1\}^{\mathbb Z^+}$. Fix $0<\theta<1$
and let $\{h_\omega:\omega\in\Omega\}$ be a continuously-parameterized
family of antisymmetric monotonic elements of $\mathcal
C_\theta(\Sigma)$, where a function is \emph{antisymmetic} if it
satisfies $h(\bar x)=-h(x)$ for $x\in\Sigma$, where $\bar x_i=1-x_i$. A
function will be called \emph{monotonic} if it satisfies $h(x)\le h(y)$
whenever $x\preceq y$, where $x\preceq y$ means $x_i\le y_i$ for each
$i$.

We will assume that $\|h_\omega\|_\infty<a<1/2$ for all
$\omega\in\Omega$. We then define a continuously parameterized family of
elements $g_\omega$ of $\mathcal W_\theta$ by
\begin{align*}
  g_\omega(1x)&=\textstyle{\frac12}+h_\omega(x)\\
  g_\omega(0x)&=\textstyle{\frac12}-h_\omega(x).
\end{align*}
From the choice of $g_\omega$, we see that $P_\omega\mathbf 1=\mathbf 1$
for all $\omega$ so that $R^*=1$ and $\lambda(\omega)=1$ for a.e.
$\omega$ and hence $\kappa(\omega)=\theta$ for a.e. $\omega$ by
Theorem \ref{thm:tospec}. 

One can verify that the $P_\omega$ map antisymmetric functions to
antisymmetric functions and monotonic functions to monotonic functions.
Following Liverani \cite{Liverani} we define a cone $\mathcal
K_a=\{f\colon f(x)>0,\ \forall x; f(x)/f(y)\le e^{ad_\theta(x,y)},\
\forall x,y\}$. For suitably large $a$, there is $a'<a$ such that
$P_\omega(\mathcal K_a)\subset \mathcal K_{a'}$. Since $\mathbf 1$ is a
fixed point the theory of cones guarantees that if $f$ is a positive
function in $\mathcal C_\theta$, then $P^{(n)}_\omega f$ converges at an
exponential rate to a constant uniformly in $\omega$. In particular an
antisymmetric function $f$ can be written as the difference of two
positive functions: $f_1-f_2$. Since $P^{(n)}_\omega f_1$ converges
exponentially fast to a constant $C_1(\omega)$ and $P^{(n)}_\omega f_2$
converges exponentially to a constant $C_2(\omega)$, the fact that
$P^{(n)}_\omega f$ remains antisymmetric implies that
$C_1(\omega)=C_2(\omega)$. It follows that $P^{(n)}_\omega f$ converges
at an exponential rate to 0 uniformly over $\omega\in\Omega$.

Choosing $f(x)=\mathbf 1_{[1]}-\mathbf 1_{[0]}$, $f$ is both monotone and
antisymmetric. It follows that $P^{(n)}_\omega f$ decays exponentially.
We are able to give a lower bound on the decay rate that guarantees the
presence of non-trivial exceptional spectrum. Specifically, using the fact that
$g_\omega(0x)+g_\omega(1x)=1$, we have
\begin{align*}
&P_\omega f(1111\ldots) =g_\omega(1111\ldots)f(1111\ldots)+
g_\omega(0111\ldots)f(0111\ldots)\\
&=
g_\omega(1111\ldots)f(1111\ldots)+(1-g_\omega(1111\ldots))f(0000\ldots)\\
&=(2g_\omega(1111\ldots)-1)f(1111\ldots).
\end{align*}

If the $h_\omega$ are chosen in such a way that $g_\omega(1111\ldots)$
is uniformly close to $1$ as $\omega$ varies then we will ensure that
there is non-trivial exceptional spectrum.

\end{example}

\section*{Acknowledgements}
GF and SL acknowledge support by the Australian Research
Council Discovery Project DP0770289. 
AQ acknowledges partial support from the Natural
Sciences and Engineering Research Council of Canada and support while visiting the University of New South Wales from the Australian Research Council Centre of Excellence for Mathematics and Statistics of Complex Systems (MASCOS).

\bibliographystyle{plain}
\bibliography{FLQ2}
\vspace{5mm}

\end{document}